\documentclass[a4paper,11pt]{article}

\usepackage{amsmath,amssymb,amsthm}
\usepackage[english]{babel}
\usepackage{enumitem}
\usepackage[top=2.5cm,bottom=2.5cm,left=2.5cm,right=2.5cm]{geometry}
\usepackage{mathrsfs}
\usepackage{tikz}
\usepackage{url}
\usepackage{xcolor}

\makeatletter
\newtheorem*{rep@theorem}{\rep@title}
\newcommand{\newreptheorem}[2]{%
    \newenvironment{rep#1}[1]{%
        \def\rep@title{#2 \ref{##1}}%
        \begin{rep@theorem}
    }%
    {\end{rep@theorem}}
}
\makeatother

\newtheorem{thm}{Theorem}[subsection]
\newreptheorem{thm}{Theorem}
\newtheorem{prop}[thm]{Proposition}
\newtheorem{prelm}{Lemma}[section]
\newtheorem{lm}[thm]{Lemma}
\newtheorem{res}[thm]{Result}
\newtheorem{crl}[thm]{Corollary}
\newreptheorem{crl}{Corollary}

\theoremstyle{definition}
\newtheorem{prob}[thm]{Open Problem}
\newtheorem{rmk}[thm]{Remark}
\newtheorem{predf}[prelm]{Definition}
\newtheorem{df}[thm]{Definition}
\newtheorem{prent}[prelm]{Notation}
\newtheorem{nt}[thm]{Notation}
\newtheorem{config}[thm]{Configuration}
\newtheorem{constr}[thm]{Construction}

\newcommand{\NN}{\mathbb{N}}
\newcommand{\NNnot}{\NN\setminus\{0\}}
\newcommand{\RR}{\mathbb{R}}
\newcommand{\FF}{\mathbb{F}}

\newcommand{\n}{N} 
\newcommand{\len}{n} 
\newcommand{\red}{r} 

\newcommand{\pg}{\textnormal{PG}}
\newcommand{\ag}{\textnormal{AG}}
\newcommand{\vspan}[1]{\left \langle #1 \right \rangle}
\newcommand{\vspanbig}[1]{\big \langle #1 \big \rangle}
\newcommand{\vspanq}[1]{\vspan{#1}_q}
\newcommand{\vspanqr}[1]{\vspan{#1}_{(q')^{\rho+1}}}
\newcommand{\vspanqm}[1]{\vspan{#1}_{(q')^m}}

\newcommand{\satbound}{s_q}
\newcommand{\lenfunc}{\ell_q}
\newcommand{\covden}{\mu_q}
\newcommand{\covdeninf}{\overline{\mu}_q}
\newcommand{\covdensup}{\mu^*_q}
\newcommand{\Sat}{\mathcal{S}}
\newcommand{\SubSat}{\mathfrak{B}}
\newcommand{\Frame}{\mathcal{F}}
\newcommand{\BigSub}{\mathcal{A}}
\newcommand{\Sub}{\mathcal{B}}
\newcommand{\HyperSub}{\mathcal{C}}
\newcommand{\ppointsSub}{\mathcal{P}_\HyperSub}
\newcommand{\ppointsSubIndex}[1]{\mathcal{P}_{\HyperSub_{#1}}}
\newcommand{\llinesSub}{\mathcal{L}_\HyperSub}
\newcommand{\llinesSubIndex}[1]{\mathcal{L}_{\HyperSub_{#1}}}
\newcommand{\AffineHyperSub}{\mathcal{D}}
\newcommand{\ppointsAff}{\mathscr{P}_\AffineHyperSub}
\newcommand{\llinesAff}{\mathscr{L}_\AffineHyperSub}
\newcommand{\ppointsX}{\mathcal{P}_X}
\newcommand{\llinesX}{\mathcal{L}_X}
\newcommand{\Flower}{\mathcal{F}}
\newcommand{\ppi}{\mathfrak{P}}

\newcommand{\pr}[3]{\textnormal{proj}_{#1,#2}^{#3}}
\newcommand{\prArg}[4]{\textnormal{proj}_{#1,#2}^{#3}(#4)}
\newcommand{\sh}[3]{\textnormal{shad}_{#1,#2}^{#3}}
\newcommand{\shArg}[4]{\textnormal{shad}_{#1,#2}^{#3}(#4)}
\newcommand{\restr}[2]{{#1}_{|#2}}
\newcommand{\fn}[2]{\lceil#1\rfloor^{(#2)}}

\renewcommand{\geq}{\geqslant}
\renewcommand{\leq}{\leqslant}
\renewcommand{\rho}{\varrho}
\renewcommand{\dim}[1]{\textnormal{dim}\left(#1\right)}

\title{Constructing saturating sets in projective spaces using subgeometries}
\author{Lins Denaux \\ {\it Ghent University}}
\date{}

\begin{document}
	
	\maketitle
	
	\begin{abstract}
	    A $\rho$-saturating set of $\pg(\n,q)$ is a point set $\Sat$ such that any point of $\pg(\n,q)$ lies in a subspace of dimension at most $\rho$ spanned by points of $\Sat$.
	    It is generally known that a $\rho$-saturating set of $\pg(\n,q)$ has size at least $c\cdot\rho\,q^\frac{\n-\rho}{\rho+1}$, with $c>\frac{1}{3}$ a constant.
		
		Our main result is the discovery of a $\rho$-saturating set of size roughly $\frac{(\rho+1)(\rho+2)}{2}q^\frac{\n-\rho}{\rho+1}$ if $q=(q')^{\rho+1}$, with $q'$ an arbitrary prime power.
		The existence of such a set improves most known upper bounds on the smallest possible size of $\rho$-saturating sets if $\rho<\frac{2\n-1}{3}$.
		As saturating sets have a one-to-one correspondence to linear covering codes, this result improves existing upper bounds on the length and covering density of such codes.
		
		To prove that this construction is a $\rho$-saturating set, we observe that the affine parts of $q'$-subgeometries of $\pg(\n,q)$ having a hyperplane in common, behave as certain lines of $\ag\big(\rho+1,(q')^\n\big)$.
	    More precisely, these affine lines are the lines of the linear representation of a $q'$-subgeometry $\pg(\rho,q')$ embedded in $\pg\big(\rho+1,(q')^\n\big)$.
	\end{abstract}
	
	
	
	\section{Motivation}\label{Sect_Mot}
	
	The main topic of this article are $\rho$-saturating sets of the Desarguesian projective space $\pg(\n,q)$.
	A $\rho$-saturating set of $\pg(\n,q)$ is a point set $\Sat$ such that any point of $\pg(\n,q)$ lies in a subspace of dimension at most $\rho$ spanned by points of $\Sat$.
	These combinatorial structures are very interesting from a coding-theoretical point of view, since they have a one-to-one correspondence to linear covering codes with covering radius $\rho+1$ (see Subsection \ref{Subsec_CovCod}).
	The existence of a small $\rho$-saturating set implies the existence of a $(\rho+1)$-covering code with small length, a property which is generally desired.
	
	\section{Preliminaries}\label{Sect_Prel}
	
	Throughout this work, we assume $\n\in\NNnot$ and $\rho\in\{0,1,\dots,\n\}$.
	Whenever we shift our geometrical perspective to its coding theoretical counterpart, we assume $\red\in\NN\setminus\{0,1\}$ and $R\in\{1,\dots,\red\}$ (see Subsection \ref{Subsec_CovCod}).
	Furthermore, we assume $q$ and $q'$ to be arbitrary prime powers.
	For the purpose of this article, it is useful to keep in mind that the assumption $q=(q')^{\rho+1}$ (equivalently, $q=(q')^R$) will often be made.
	
	\bigskip
	We will denote the Galois field $\textnormal{GF}(q)$ of order $q$ by $\FF_q$ and the Desarguesian projective space of (projective) dimension $\n$ over $\FF_q$ by $\pg(\n,q)$.
	By omitting a hyperplane in $\pg(\n,q)$, we naturally obtain the Desarguesian affine space of dimension $\n$ over $\FF_q$, which we will denote by $\ag(\n,q)$.
	Furthermore, define the value
	\[
	    \theta_\n := \frac{q^{\n+1}-1}{q-1}\textnormal{,}
	\]
	which equals the number of points in $\pg(\n,q)$.
	
	\begin{predf}
		Let $m\in\NNnot$.
		A \emph{frame} of the projective geometry $\pg(m,q)$ is a set of $m+2$ points of which no $m+1$ points are contained in a hyperplane.
	\end{predf}
	
	The notion of a $q'$-subgeometry will be a key concept throughout this article.
	For a detailed description on subgeometries, see \cite[p.~103]{Hirschfeld}.
	
	\begin{predf}\label{Def_Subgeometry}
		Let $m,m'\in\NN$, $m'\leq m$, and suppose $q=(q')^{\rho+1}$.
		An $m'$\emph{-dimensional} $q'$\emph{-subgeometry} $\Sub$ of $\pg(m,q)$ is a set of subspaces (points, lines, \ldots, $(m'-1)$-dimensional subspaces) of $\pg(m,q)$, together with the incidence relation inherited from $\pg(m,q)$, such that $\Sub$ is isomorphic to $\pg(m',q')$.
		
		If $m'=1$ or $m'=2$, we will often call $\Sub$ a \emph{$q'$-subline} or a \emph{$q'$-subplane} of $\pg(m,q)$, respectively.
		Moreover, we will denote the $m'$-dimensional subspace of $\pg(m,q)$ spanned by the points of $\Sub$ by $\vspanq{\Sub}$.
		If $m'=m$, we will omit the dimension and simply call $\Sub$ a $q'$-subgeometry of $\pg(m,q)$.
		Lastly, whenever $q'$ is clear from context, the prefix `$q'$-' will often be omitted.
	\end{predf}
	
	If $\rho=1$, a $q'$-subgeometry is obviously better known under the name \emph{Baer subgeometry}, by far the most studied subgeometry of projective spaces.
	One can find a short survey on Baer sublines and Baer subplanes in \cite{BarwickEbert} and on general Baer subgeometries in \cite{Bruen}.
	
	Although a lot is known about $q'$-subgeometries, we will only use a few properties concerning these structures, one of which is the following.
	The proof is done by considering the underlying vector space of the projective geometry.
	
	\begin{prelm}[{\cite[Theorems $2.6$ and $2.8$]{BarwickEbert} and \cite[Lemma $1$]{Bruen}}]
		Let $m\in\NNnot$ and let $q$ be square.
		For each frame of $\pg(m,q)$, there exists a unique Baer subgeometry containing each point of the frame.
	\end{prelm}
	
	Using the exact same arguments as used in \cite{BarwickEbert,Bruen}, one can easily generalise the proof of the lemma above to arbitrary subgeometries:
	
	\begin{prelm}\label{Lm_UniqueSubThroughFrame}
		Let $m\in\NNnot$ and suppose $q=(q')^{\rho+1}$.
		For each frame of $\pg(m,q)$, there exists a unique $q'$-subgeometry containing each point of the frame.
	\end{prelm}
	
	Basically, every choice of frame in $\pg(m,q)$ and subfield in $\FF_q$ results in finding a unique subgeometry defined over that subfield and containing the frame.
	
	\bigskip
	From this, we can deduce the following property.
	
	\begin{prelm}\label{Lm_UniqueSubThroughHyperSubAndLine}
		Let $m\in\NNnot$ and suppose $q=(q')^{\rho+1}$.
		Consider an $(m-1)$-dimensional $q'$-subgeometry $\HyperSub$ of $\pg(m,q)$ and define $\Sigma:=\vspanq{\HyperSub}$.
		Let $\mathfrak{L}$ be a $q'$-subline of $\pg(m,q)$ having a point in common with $\HyperSub$ and spanning a line $\ell\nsubseteq\Sigma$.
		Then there exists a unique $q'$-subgeometry containing both $\HyperSub$ and $\mathfrak{L}$.
	\end{prelm}
	\begin{proof}
		Let $P$ be the unique point in $\HyperSub\cap\mathfrak{L}$.
		Consider a frame $\Frame_\HyperSub:=\{P,P_1,P_2,\dots,P_m\}$ of $\HyperSub$.
		Naturally, $\Frame_\HyperSub$ is a frame of $\Sigma$ as well, uniquely determining, by Lemma \ref{Lm_UniqueSubThroughFrame}, the $(m-1)$-dimensional $q'$-subgeometry $\HyperSub$.
		For any two distinct points $Q_1,Q_2$ of $\mathfrak{L}\setminus\{P\}$, the set $\Frame:=\{P_1,P_2,\dots,P_m\}\cup\{Q_1,Q_2\}$ is a frame of $\pg(m,q)$.
		Hence, by Lemma \ref{Lm_UniqueSubThroughFrame}, there exists a unique $m$-dimensional $q'$-subgeometry $\Sub$ containing each point of $\Frame$.
		
		Note that the set $\{P_1,P_2,\dots,P_m\}$ spans an $(m-1)$-dimensional $q'$-subgeometry $\HyperSub'\subseteq\Sub$ and that the set $\{Q_1,Q_2\}$ spans a $q'$-subline $\mathfrak{L}'\subseteq\Sub$.
		As $\HyperSub'$ (spanning $\Sigma$) plays the role of an $(m-1)$-dimensional subspace of $\Sub$, and $\mathfrak{L}'$ (spanning $\ell$) plays the role of a line of $\Sub$, by Grassmann's identity, these two objects have a point of $\Sub$ in common, necessarily equal to $\Sigma\cap \ell=P$.
		This means that $\Frame_\HyperSub\subseteq\HyperSub'$ and thus, by Lemma \ref{Lm_UniqueSubThroughFrame}, $\HyperSub=\HyperSub'$.
		The same holds for $\mathfrak{L}'$ and $\mathfrak{L}$, as both contain the frame $\{P,Q_1,Q_2\}$.
	\end{proof}
	
	Now we introduce the most important definition of this work.
	
	\begin{predf}
		Let $\Sat$ be a point set of $\pg(\n,q)$.
		\begin{enumerate}
			\item A point $P\in\pg(\n,q)$ is said to be $\rho$-\emph{saturated} by $\Sat$ (or, conversely, the set $\Sat$ $\rho$-\emph{saturates} $P$) if there exists a subspace through $P$ of dimension at most $\rho$ that is spanned by points of $\Sat$.
			If $\rho$ is clear from context, the prefix `$\rho-$' is often omitted.
			\item The set $\Sat$ is a $\rho$-\emph{saturating set} of $\pg(\n,q)$ if $\rho$ is the least integer such that all points of $\pg(\n,q)$ are $\rho$-saturated by $\Sat$.
		\end{enumerate}
	\end{predf}
	
	As reasoned in Section \ref{Sect_Mot} (see also Subsection \ref{Subsec_CovCod}), it is justifiable to study small $\rho$-saturating sets of $\pg(\n,q)$, as these objects give rise to covering codes with good properties.
	In light of this, we will adopt the following notation, which is widely used in the literature (e.g.\ \cite{BartoliEtAl1,BartoliEtAl2,DavydovEtAl2,DavydovEtAl3}).
	
	\begin{prent}
		$\satbound(\n,\rho):=\min\left\{|\Sat|:\Sat\textnormal{ is a }\rho\textnormal{-saturating set of }\pg(\n,q)\right\}$.
	\end{prent}
	
	The main research concerning saturating sets focuses on finding small upper bounds on $\satbound(\n,\rho)$.
	
	\section{Outline and main results}\label{Sect_Outline}
	
	General preliminaries can be found in Section \ref{Sect_Prel}.
	Section \ref{Sect_CovCod} formalises the correspondence between saturating sets and covering codes.
	In light of this, relevant definitions and notation within a coding theoretical context are introduced.
	Furthermore, we state a quasi-trivial lower bound on the size of a $\rho$-saturating set, which naturally gives rise to the main research goal of this topic (see Open Problem \ref{Prob_Main}).
	The core of this work can be found in Section \ref{Sect_Subgeometric}, although Section \ref{Sect_PointLine} in itself presents an interesting, stand-alone result.
	
	\bigskip
	Section \ref{Sect_PointLine} describes an isomorphism between two point-line geometries.
	One is the linear representation $T^*(\AffineHyperSub_{\rho,m,q'})$ of a subgeometry $\AffineHyperSub_{\rho,m,q'}\cong\pg(\rho,q')\subseteq\pg\big(\rho+1,(q')^m\big)$.
	The other is a newly introduced point-line geometry $Y(\rho,m,q')$, embedded in $\pg\big(m,(q')^{\rho+1}\big)$, of which the lines are affine parts of subgeometries isomorphic to $\pg(m,q')$ (see Definition \ref{Def_PointLineY}).
	
	\begin{repthm}{Thm_SubgeometriesAreAffineLines}
		Let $m\in\NNnot$.
		Then the point-line geometries $Y(\rho,m,q')$ and $T^*(\AffineHyperSub_{\rho,m,q'})$ are isomorphic.
	\end{repthm}
	
	Consequently, one can transfer natural notions of parallelism and independence of concurrent lines from $T^*(\AffineHyperSub_{\rho,m,q'})$ to $Y(\rho,m,q')$ (see Subsection \ref{Subsec_ParallelismIndependence}).
	As a side note, we make the reader aware of the existence of an explicit isomorphism between this newly introduced point-line geometry $Y(\rho,m,q')$ and the point-line geometry $X(\rho,m,q')$ introduced by De Winter, Rottey and Van de Voorde \cite{DeWinterRotteyVandeVoorde} (see Subsection \ref{Subsec_IsomorphismXFieldRed}).
	
	\bigskip
	Section \ref{Sect_Subgeometric} discusses the main result of this work by presenting a general upper bound on $\satbound(\n,\rho)$, $q=(q')^{\rho+1}$. This is obtained by constructing a $\rho$-saturating set of $\pg(\n,q)$ as a mix of several distinct, partially overlapping $q'$-subgeometries.
	The technique used to prove the saturation property of this construction relies on the results obtained in Section \ref{Sect_PointLine}.
	
	Although a precise and extensive upper bound is given by Theorem \ref{Thm_MainUpperBoundSubgeometric}, we present the following consequence as our main result, which is slightly weaker but far easier to comprehend.
	
	\begin{repthm}{Thm_MainUpperBoundSubgeometricRhoLargerThan1}
	    Let $1<\rho<\n$ and let $q=(q')^{\rho+1}$ for any prime power $q'$.
		Then
		\[
		    \satbound(\n,\rho) \leq \frac{(\rho+1)(\rho+2)}{2}(q')^{\n-\rho} + \rho(\rho+1)\frac{(q')^{\n-\rho}-1}{q'-1}\textnormal{.}
		\]
	\end{repthm}
	
	Translating the result above in coding theoretical terminology (see Subsection \ref{Subsec_CovCod}), one obtains the following.
	
	\begin{repcrl}{Crl_MainUpperBoundSubgeometricRLargerThan2}
	    Let $2<R<\red$ and let $q=(q')^R$ for any prime power $q'$.
	    Then
	    \[
		    \lenfunc(\red,R) \leq \frac{R(R+1)}{2}(q')^{\red-R} + (R-1)R\frac{(q')^{\red-R}-1}{q'-1}\textnormal{.}
		\]
		For any infinite family of covering codes of length equal to the upper bound above, the following holds for its asymptotic covering density:
		\[
		    \covdeninf(R)<\frac{\big((R-1)R\big)^R}{R!}\left(1+\frac{1}{q'}+\dots+\frac{1}{(q')^{R-1}}\right)^R<\left(e(R-1)\frac{q-1}{q-(q')^{R-1}}\right)^R\textnormal{,}
		\]
		with $e\approx2.718...$ being Euler's number.
	\end{repcrl}
	
	Plenty of extensive research has already been done concerning the topic of saturating sets and covering codes.
	Therefore, Section \ref{Sect_Comp} provides a careful comparison between our main results and relevant known results from the literature.
	
	\section{Covering codes and the research goal}\label{Sect_CovCod}
	
	Based on the way to approach this topic of research, the literature is divided.
	On the one hand, one can observe the topic geometrically by analysing small $\rho$-saturating sets of $\pg(\n,q)$.
	On the other hand, one can convert this geometrical point of view to a coding theoretical one by investigating \emph{covering codes} of small length.
	
	\subsection{Translation to covering codes}\label{Subsec_CovCod}
	
	In this subsection, we aim to formalise the correspondence between saturating sets and covering codes described in Section \ref{Sect_Mot}.
	
	\begin{df}
	    A $q$-ary linear code of length $\len$ and codimension (redundancy) $\red$ is said to have \emph{covering radius} $R$ if $R$ is the least integer such that every vector of $\FF_q^\len$ lies within Hamming distance\footnote{The Hamming distance between two vectors of $\FF_q^\len$ equals the number of positions in which they differ.} $R$ of a codeword.
	    Such a code will be called an $[\len,\len-\red]_qR$ code.
	\end{df}
	
	Whenever linear codes are investigated with the goal of optimising the length or (co)dimension with respect to the covering radius, such codes are often called \emph{covering codes}.
	These type of $q$-ary linear codes have a wide range of applications; for a description of several examples of such applications, see \cite[Section $1$]{DavydovEtAl1}.
	
	Suppose that $\Sat$ is a point set of $\pg(\red-1,q)$ of size $\len$ and let $H$ be a $q$-ary $(\red\times\len)$-matrix with the homogeneous coordinates of the points of $\Sat$ as columns.
	Then $\Sat$ is an $(R-1)$-saturating set of $\pg(\red-1,q)$ if and only if $H$ is a parity check matrix of an $[\len,\len-\red]_qR$ code.
	This describes a one-to-one correspondence between saturating sets of projective spaces and linear covering codes.
	More specifically, any $\rho$-saturating set $\Sat$ of $\pg(\n,q)$ corresponds to an $[\len,\len-\red]_qR$ code with
	\[
	    \len=|\Sat|\textnormal{,}\qquad\red=\n+1\quad\textnormal{and}\quad R=\rho+1\textnormal{.}
	\]
	
	Due to this correspondence, the problem of finding small $\rho$-saturating sets in $\pg(\n,q)$ can be translated to finding $[\len,\len-\red]_qR$ codes of small length.
	In light of this, we adopt the notation of the \emph{length function} $\lenfunc(\red,R)$.
	
	\begin{df}[\cite{BrualdiEtAl,CohenEtAl}]
	    The \emph{length function} $\lenfunc(\red,R)$ is the smallest length of a $q$-ary linear code with covering radius $R$ and codimension $\red$.
	\end{df}
	
	Note that
	\[
	    \lenfunc(\red,R)=\satbound(\red-1,R-1)\textnormal{.}
	\]
	
	From a coding theoretical perspective, it is interesting to analyse the extent to which the spheres of radius $R$ centered at the codewords of an $[\len,\len-\red]_qR$ code $\mathcal{C}$ overlap.
	This is done by investigating the \emph{covering density} $\covden(\len,\red,\mathcal{C})$ of the code $\mathcal{C}$, which equals the ratio of the total volume of these $q^{\len-\red}$ spheres to the volume of the space $\FF_q^\len$ (see e.g.\ \cite{BartoliEtAl1,BartoliEtAl2,BartoliEtAl3,Davydov1,Davydov2,DavydovEtAl1}):
	\[
	    \covden(\len,\red,\mathcal{C}):=\frac{1}{q^\len}\left(q^{\len-\red}\sum_{i=0}^R(q-1)^i\binom{\len}{i}\right)=\frac{1}{q^\red}\sum_{i=0}^R(q-1)^i\binom{\len}{i}\geq1\textnormal{.}
	\]
	Note that the latter inequality is sharp if and only if $\mathcal{C}$ is a perfect code.
	
	Analogous to the work of Davydov et al.\ \cite{DavydovEtAl1}, for a given $R\in\{1,\dots,\len\}$ and fixed prime power $q$, we will call an infinite sequence of $q$-ary linear $[\len,\len-\red_\len]_qR$ codes $\mathcal{C}_\len$ an \emph{infinite family of covering codes}, and denote such a family with $\mathcal{A}_{R,q}$.
	Given such an infinite family of covering codes $\mathcal{A}_{R,q}$, the asymptotic behaviour of the covering density when $\len$ tends to infinity is a topic of investigation.
	In light of this, we define its \emph{asymptotic covering densities}
	\begin{itemize}
	    \item $\covdeninf(R,\mathcal{A}_{R,q}):=\liminf\limits_{\len\rightarrow\infty}\covden(\len,R,\mathcal{C}_\len)$, and
	    \item $\covdensup(R,\mathcal{A}_{R,q}):=\limsup\limits_{\len\rightarrow\infty}\covden(\len,R,\mathcal{C}_\len)$.
	\end{itemize}
	If the infinite family of covering codes $\mathcal{A}_{R,q}$ is clear from context, we will write $\covdeninf(R)$ (respectively $\covdensup(R)$) instead of $\covdeninf(R,\mathcal{A}_{R,q})$ (respectively $\covdensup(R,\mathcal{A}_{R,q})$).
	
	\bigskip
	As the authors of \cite{DavydovEtAl1} point out, given an infinite family $\mathcal{A}_{R,q}$ for which $r_{n+1}<r_n$ for some $n$, one can replace $\mathcal{C}_{n+1}$ by any $1$-extension of $\mathcal{C}_n$ to obtain a code with a better covering density.
	Hence one may assume that the sequence of codimensions $r_n$ of $\mathcal{A}_{R,q}$ is non-decreasing.
	In light of this, a code $\mathcal{C}_n$ is called a \emph{supporting code} of $\mathcal{A}_{R,q}$ if $r_n>r_{n-1}$ (and a \emph{filling code} otherwise).
	
	Davydov et al.\ \cite{DavydovEtAl1} introduce and solve two open problems.
	The first open problem concerns the search for an infinite family of covering codes $\mathcal{A}_{R,q}$ for which $\covdensup(R,\mathcal{A}_{R,q})=\mathcal{O}(q)$; in this case, the corresponding family $\mathcal{A}_{R,q}$ is said to be \emph{optimal} \cite[Open Problem $1$]{DavydovEtAl1}.
	The authors note that, in order to solve this open problem, it suffices to find a solution to the following:
	
	\begin{prob}[{\cite[Open Problem $2$]{DavydovEtAl1}}]
	    For any covering radius $R\geq2$, construct $R$ infinite families of covering codes $\mathcal{A}_{R,q}^{(0)},\mathcal{A}_{R,q}^{(1)},\dots,\mathcal{A}_{R,q}^{(R-1)}$ such that for each $\gamma=0,1,\dots,R-1$ the supporting codes of $\mathcal{A}_{R,q}^{(\gamma)}$ are $[\len_u,\len_u-\red_u]_qR$ codes with codimension $\red_u=Ru+\gamma$ and length $\len_u=f_q^{(\gamma)}(\red_u)$ with $f_q^{(\gamma)}(\red)=\mathcal{O}\left(q^\frac{\red-R}{R}\right)$ for any $u\geq u_0$, with $u_0$ a constant which may depend on the family.
	\end{prob}
	
	The authors of \cite{DavydovEtAl1} managed to solve this open problem for arbitrary covering radius $R\geq2$ and $q=(q')^R$ (see \cite[Section $4$]{DavydovEtAl1} and Subsection \ref{Subsec_KnownApproaches}).
	We managed to do the same, in a somewhat more effective way, as we could take $u_0=1$ independent of the infinite family of covering codes, and in most cases found a substantially smaller polynomial function $f_q^{(\gamma)}(\red)$ in $q$ of which the leading coefficient is quadratic in $R$.
	As a consequence, the infinite families of covering codes we obtained have a generally improved asymptotic covering density (see Subsection \ref{Subsec_KnownApproaches}).
	
	\subsection{A lower bound on $\boldsymbol{\satbound(\n,\rho)}$ and $\boldsymbol{\lenfunc(\red,R)}$}\label{Subsec_LowBound}
	
	In order to know which saturating sets are viewed as being `small', we will be guided by the following lower bound on the size of arbitrary $\rho$-saturating sets.
	Several variants of this bound were already known in the literature \cite{BartoliEtAl1,BartoliEtAl2,BartoliEtAl3,DavydovEtAl1,DavydovEtAl2,DavydovEtAl3}, but some only state the bound for specific values of $\rho$, while others describe an approximate lower bound for large values of $q$.
	
	\begin{prop}\label{Prop_LowerBound}
		Let $\Sat$ be a $\rho$-saturating set of $\pg(\n,q)$, $\rho\leq \n$.
		Then
		\[
		    |\Sat|>\frac{\rho+1}{e}\cdot q^{\frac{\n-\rho}{\rho+1}}+\frac{\rho}{2}\textnormal{,}
		\]
		with $e\approx2.718...$ being Euler's number.
	\end{prop}
	\begin{proof}
		Note that $|\Sat|\geq\rho+1$.
		Indeed, if this would not be the case, all points of $\Sat$ would span a subspace of dimension at most $\rho-1<\n$.
		In such a situation, it is impossible for $\Sat$ to saturate all points of the $\n$-dimensional projective geometry $\pg(\n,q)$, a contradiction.
		
		Hence, we can consider the set $\Pi_{\leq\rho}$ of all subspaces spanned by $\rho+1$ distinct points of $\Sat$; these subspaces are each of dimension at most $\rho$.
		As $\Sat$ saturates $\pg(\n,q)$, we know that $\Pi_{\leq\rho}$ has to cover the latter, thus
		\[
		    \binom{|\Sat|}{\rho+1}\theta_{\rho}\geq\theta_\n\textnormal{.}
		\]
		Expanding the binomial above and rearranging the inequality, we get
		\begin{equation}\label{Eq_Approx}
		    |\Sat|\left(|\Sat|-1\right)\left(|\Sat|-2\right)\cdots\left(|\Sat|-\rho\right)\geq(\rho+1)!\cdot\frac{\theta_\n}{\theta_\rho}\geq(\rho+1)!\cdot q^{\n-\rho}\textnormal{,}
		\end{equation}
		the last inequality being valid if and only if $\rho\leq\n$.
		
		Note that the map $f:\NNnot\rightarrow\RR:\n\mapsto\frac{\sqrt[\n]{\n!}}{\n}$ is strictly decreasing, with $\lim\limits_{\n\rightarrow\infty}f(\n)=\frac{1}{e}$.
		Therefore, we know that $f(\n)>\frac{1}{e}$ for all $\n\in\NNnot$, or, equivalently, $\sqrt[\n]{\n!}>\frac{\n}{e}$.
		Combining this with \eqref{Eq_Approx}, after taking the $(\rho+1)^\textnormal{th}$ root of the left- and right-hand side, we obtain
		\[
		    \sqrt[\rho+1]{|\Sat|\left(|\Sat|-1\right)\left(|\Sat|-2\right)\cdots\left(|\Sat|-\rho\right)}>\frac{\rho+1}{e}\cdot q^{\frac{\n-\rho}{\rho+1}}\textnormal{.}
		\]
		Applying the AM-GM inequality to the left-hand side finishes the proof.
	\end{proof}
	
	Roughly speaking, Proposition \ref{Prop_LowerBound} implies that
	\begin{equation}\label{Eq_MainGoal}
	    \satbound(\n,\rho)\geq c\cdot\rho\,q^\frac{\n-\rho}{\rho+1}\textnormal{,}\qquad\textnormal{or, equivalently,}\qquad\lenfunc(\red,R)\geq c\cdot R\,q^\frac{\red-R}{R}\textnormal{,}
	\end{equation}
	for any $\n$, $\rho$ and $q$ (equivalently, for any $\red$, $R$ and $q$), where $c>\frac{1}{3}$ is a constant independent of these parameters.
	Naturally, researchers aim to prove that \eqref{Eq_MainGoal} is sharp by constructing small $\rho$-saturating sets of $\pg(\n,q)$ or, equivalently, constructing $[\len,\len-\red]_qR$ covering codes of small length.
	This gives rise to Open Problem \ref{Prob_Main}.
	
	\section{Comparison with relevant known results}\label{Sect_Comp}
	
	In Subsection \ref{Subsec_LowBound}, we deduced a general lower bound on the size of a saturating set (equivalently, on the length of a covering code).
	This gives rise to the following.
	
	\begin{prob}\label{Prob_Main}
	    Find a value $c>0$, independent of $q$ (and preferably independent of $\n$ and $\rho$ as well), such that
	    \[
	        \satbound(\n,\rho)\leq c\cdot\rho\,q^\frac{\n-\rho}{\rho+1}\textnormal{,}
	    \]
	    or, equivalently, find a value $c>0$, independent of $q$ (and preferably independent of $\red$ and $R$ as well), such that
	    \[
	        \lenfunc(\red,R)\leq c\cdot R\,q^\frac{\red-R}{R}\textnormal{.}
	    \]
	\end{prob}
	
	With the exception of Remark \ref{Rmk_AlmostSolved}, all mentioned results within this section solve the open problem above for specific values of $\n$, $\rho$ and $q$ (equivalently, $\red$, $R$ and $q$), some in a more effective way than others.
	To give an overview, Open Problem \ref{Prob_Main} is solved if
	\begin{enumerate}
	    \item $\n+1\equiv0\pmod{\rho+1}$ (equivalently, $\red\equiv0\pmod{R}$); see Subsection \ref{Subsec_KnownDimension}.
	    \item $q=(q')^{\rho+1}$ (equivalently, $q=(q')^R$); see Subsection \ref{Subsec_KnownApproaches}.
	    \item $\n+1\equiv s\frac{\rho+1}{\rho'+1}\pmod{\rho+1}$ and $q=(q')^{\rho'+1}$, with $\rho'+1\mid\rho+1$ and $s\in\{1,2,\dots,\rho'\}$ (equivalently, $\red\equiv s\frac{R}{R'}\pmod{R}$ and $q=(q')^{R'}$, with $R'\mid R$ and $s\in\{1,2,\dots,R'-1\}$); see Remark \ref{Rmk_Section7}.
	\end{enumerate}
	
	\begin{rmk}\label{Rmk_AlmostSolved}
	    Some results present upper bounds that are slightly larger than the desired one described in Open Problem \ref{Prob_Main}.
    	More specifically, the authors of articles \cite{BartoliEtAl1,BartoliEtAl2,BartoliEtAl3,DavydovEtAl2,DavydovEtAl4}, some with the aid of computer searches, present upper bounds on $\satbound(\n,\rho)$, $\rho\in\{1,2\}$, of the following form:
    	\[
    	    \satbound(\n,\rho)\leq c\cdot q^\frac{\n-\rho}{\rho+1}\sqrt[\rho+1]{\ln{q}}\textnormal{,}
    	\]
    	with $c>0$ a constant independent of $\n$, and $\sqrt[\rho+1]{\ln{q}}$ a relatively small factor dependent on $q$.
    \end{rmk}
	
	One can immediately see that, if $q=(q')^{\rho+1}=(q')^R$, our main results (Theorem \ref{Thm_MainUpperBoundSubgeometricRhoLargerThan1} and Corollary \ref{Crl_MainUpperBoundSubgeometricRLargerThan2}) also solve Open Problem \ref{Prob_Main} for $c$ independent of $q$ and $\n$ (equivalently, independent of $q$ and $\red$), and linearly dependent on $\rho$ (equivalently, linearly dependent on $R$).
	In this section, we will carefully compare these new results to relevant known results from the literature.
	Depending on the setting of each of these relevant results, we will make the comparison from a geometrical point of view (Theorem \ref{Thm_MainUpperBoundSubgeometricRhoLargerThan1}) or a coding theoretical point of view (Corollary \ref{Crl_MainUpperBoundSubgeometricRLargerThan2}).
	
	\subsection{Known results with assumptions on $\boldsymbol{\n}$ and $\boldsymbol{\rho}$ (correspondingly $\boldsymbol{\red}$ and $\boldsymbol{R}$)}\label{Subsec_KnownDimension}
	
	A simple, recursive upper bound on $\satbound(\n,\rho)$ can be obtained geometrically by observing saturating sets in two disjoint subspaces spanning the ambient geometry.
	As stated in \cite[Theorem $5$]{DavydovOstergard}, the same bound arises from the direct sum construction of linear codes over a common finite field.
	
	\begin{res}[{\cite[Lemma $10$]{Ughi}}]\label{Res_RecursiveUpperBound}
		$\satbound(\n_1+\n_2+1,\rho_1+\rho_2+1)\leq\satbound(\n_1,\rho_1)+\satbound(\n_2,\rho_2)$.
	\end{res}
	
	\begin{crl}\label{Crl_TrivialUpperBound}
		Suppose that $\n+1$ is a multiple of $\rho+1$.
		Then $\satbound(\n,\rho)\leq(\rho+1)\theta_k$, with $k:=\frac{\n-\rho}{\rho+1}$.
	\end{crl}
	\begin{proof}
		By induction on $\rho$.
		If $\rho=0$, this is a trivial statement.
		Inductively using Result \ref{Res_RecursiveUpperBound}, we obtain
		\[
		\satbound(\n,\rho)\leq\satbound(\n-k-1,\rho-1)+\satbound(k,0)\leq\rho\theta_k+\theta_k\textnormal{.}\qedhere
		\]
	\end{proof}
	
	As shown in the proof of Theorem \ref{Thm_MainUpperBoundSubgeometric}, the upper bound of Theorem \ref{Thm_MainUpperBoundSubgeometric} (and hence the one of Theorem \ref{Thm_MainUpperBoundSubgeometricRhoLargerThan1} as well) does not improve the upper bound above if $\n+1$ is a multiple of $\rho+1$.
	
	Although, for $\n+1$ a multiple of $\rho+1$, Corollary \ref{Crl_TrivialUpperBound} already solves Open Problem \ref{Prob_Main} (for $c$ independent of $\n$ and $\rho$), we want to stress that better upper bounds concerning this special case are known in the literature.
	Davydov \cite[{Theorem $5.1$}]{Davydov1} and Davydov and \"Osterg\aa rd \cite[{Theorem $7$}]{DavydovOstergard} slightly improved the bound above in case $k=1$ and $\rho=1,2$, respectively.
	The constructions behind these results are commonly denoted as the `oval plus line' and `two ovals plus line' constructions; in \cite[Theorems $6.1$ and $6.2$]{DavydovEtAl1}, these bounds are generalised.
	Davydov, Marcugini and Pambianco \cite[Theorem $1$]{DavydovEtAl3} managed to generalise this `oval(s) plus line' construction to a $\rho$-saturating set in $\pg(2\rho+1,q)$.
	Using a coding-theoretical tool called `$q^m$-concatenating constructions' \cite{Davydov1,Davydov2,DavydovEtAl1,DavydovEtAl3}, they generalised their results even further and improved the upper bound depicted in Corollary \ref{Crl_TrivialUpperBound} under some minor restrictions on the parameters.
	
	\bigskip
	Furthermore, the same authors obtained the following results for the case $R$ even, $r\equiv\frac{R}{2}\pmod{R}$ and compared them with \cite[Corollary $7.2$]{DavydovEtAl1}.
	
	\begin{res}[{\cite[Theorem $2$]{DavydovEtAl3}}]\label{Res_DavydovEtAlR2}
	    Let $R\geq2$ be even.
	    Let $p$ be prime, $q=p^{2\eta}$, $\eta\geq2$, $r=tR+\frac{R}{2}$, $t\geq1$.
	    The following constructive upper bounds on the length function hold:
	    \begin{enumerate}
	        \item $\lenfunc(\red,R)\leq R\left(1+\frac{\sqrt{q}-1}{\sqrt{q}(\phi(\sqrt{q})-1)}\right)q^\frac{\red-R}{R}+R\left\lfloor q^{\frac{\red-2R}{R}-\frac{1}{2}}\right\rfloor+\frac{R}{2}f_q(\red,R)$ if $p\geq3$, and
	        \item $\lenfunc(\red,R)\leq R\left(1+\frac{1}{p}+\frac{1}{\sqrt{q}}\right)q^\frac{\red-R}{R}+R\left\lfloor q^{\frac{\red-2R}{R}-\frac{1}{2}}\right\rfloor+\frac{R}{2}f_q(\red,R)$ if $p\geq7$,
	    \end{enumerate}
	    where $\phi(q)$ is the order of the largest proper subfield of $\FF_q$ and
	    \[
	        f_q(\red,R):=\begin{cases}0&\textnormal{if }t\notin\{4,6\}\textnormal{,}\\(q+1)q^{\frac{\red-4R}{R}-\frac{1}{2}}&\textnormal{if }t\in\{4,6\}\textnormal{.}\end{cases}
	    \]
	\end{res}
	
	One can check that Corollary \ref{Crl_MainUpperBoundSubgeometricRLargerThan2} does not improve this result, nor \cite[Corollary $7.2$]{DavydovEtAl1}, for given constraints on $\red$, $R$ and $q$.
	
	\subsection{Known results assuming $\boldsymbol{q=(q')^{\rho+1}=(q')^R}$}\label{Subsec_KnownApproaches}
	
	In this subsection, we discuss some relevant known results based on the assumption that $q=(q')^{\rho+1}$ (equivalently, $q=(q')^R$).
	This assumption allows mathematicians to exploit the use of $q'$-subgeometries.
	In the literature, one can notice two main approaches for constructing saturating sets using subgeometries; we will call these two approaches the \emph{strong blocking set approach} and the \emph{mixed subgeometry approach}.
	
	\subsubsection*{The strong blocking set approach}
	
	The \emph{strong blocking set approach} is based on constructing strong blocking sets in $\pg(\n,q')$. A $(\rho+1)$\emph{-fold strong blocking set} of $\pg(\n,q')$ is a point set that meets any $\rho$-dimensional subspace in a set of points spanning said subspace.
	Although $(\rho+1)$-fold strong blocking sets were introduced in \cite[Definition $3.1$]{DavydovEtAl1}, these are also known as \emph{generator sets} (\cite[Definition $2$]{FancsaliSziklai1}) or \emph{cutting blocking sets} (\cite[Definition $3.4$]{BoniniBorello}) in case $\rho=\n-1$.
	
	Strong blocking sets directly generate saturating sets, as one can prove that $(\rho+1)$-fold strong blocking sets are $\rho$-saturating sets of the ambient geometry $\pg\big(\n,(q')^{\rho+1}\big)$ \cite[Theorem $3.2$]{DavydovEtAl1}.
	This \emph{strong blocking set approach} led to several results solving Open Problem \ref{Prob_Main}, and are often generalised using $q^m$-concatenating constructions.
	For example, the following results consider the case $\red\not\equiv0\pmod{R}$ for $R=3$ (equivalently, $\n+1\not\equiv0\pmod{\rho+1}$ for $\rho=2$).
	
	\begin{res}[{\cite[Corollary $3.9$, Theorem $5.1$]{DavydovEtAl1}}]\label{Res_FourLinesQuadric}
	    Let $t\in\NNnot$ and $\red=3t+1$.
	    Suppose that $q=(q')^3$, with $q\geq64$ if $t>1$.
	    Then
	    \[
	        \lenfunc(\red,3) \leq 4(q')^{\red-3}+4(q')^{\red-4}\textnormal{, and }\covdeninf(3)<\frac{32}{3}+\frac{32}{q'}+\frac{32}{(q')^2}-\frac{64}{3q}\textnormal{.}
	    \]
	\end{res}
	
	\begin{res}[{\cite[Theorems $3.16$ and $5.2$]{DavydovEtAl1}}]\label{Res_NinePlanes}
	    Let $t\in\NNnot$ and $\red=3t+2$.
	    Suppose that $q=(q')^3$, with $q\geq27$ if $t>1$.
	    Then
	    \[
	        \lenfunc(\red,3) \leq 9(q')^{\red-3}-8(q')^{\red-4}+4(q')^{\red-5}\textnormal{, and }\covdeninf(3)<\frac{243}{2}-\frac{324}{q'}+\frac{72}{(q')^2}\textnormal{.}
	    \]
	\end{res}
	
	Our results (Corollary \ref{Crl_MainUpperBoundSubgeometricRLargerThan2}) imply that
	\begin{equation}\label{Eq_R3}
	    \lenfunc(\red,3)\leq6\frac{(q')^{\red-2}-1}{q'-1}\textnormal{, and }\covdeninf(3)<36\left(1+\frac{1}{q'}+\frac{1}{(q')^2}\right)^3
	\end{equation}
	and hence do not improve Result \ref{Res_FourLinesQuadric} if $\red=4$ or if $\red=3t+1$, $t>1$ and $q'\geq4$.
	However, in case $\red=3t+2$, the upper bound on $\lenfunc(\red,3)$ in \eqref{Eq_R3} does improve Result \ref{Res_NinePlanes} if $q'\geq5$; the upper bound on $\covdeninf(3)$ improves Result \ref{Res_NinePlanes} if $q'\geq7$.
	
	\bigskip
	More generally, the authors of \cite{DavydovEtAl1} presented the following.
	
	\begin{res}[{\cite[Theorem $3.15$]{DavydovEtAl1}}]\label{Res_BinomUpperBound}
		Let $\n>\rho+1$ and suppose $q=(q')^{\rho+1}$ for any prime power $q'$.
		Then
		\[
		    \satbound(\n,\rho) \leq \frac{\sum_{i=0}^{\n-\rho+1}(q'-1)^i\binom{\n+1}{i}-1}{q'-1}\sim\binom{\n+1}{\rho}(q')^{\n-\rho}\textnormal{.}
		\]
	\end{res}
	
	At first sight, our results (Theorem \ref{Thm_MainUpperBoundSubgeometricRhoLargerThan1}) present a significant improvement on the bound presented in Result \ref{Res_BinomUpperBound}, as the binomial coefficient $\binom{\n+1}{\rho}$ is reduced to $\frac{(\rho+1)(\rho+2)}{2}$.
	In fact, Davydov et al.\ \cite{DavydovEtAl1} speculated that their coefficient could be improved, as they mention this as an open problem.
	However, a certain degree of nuance is needed, as the following results present better bounds in case $\red\not\equiv0\pmod{R}$, $R\geq4$, and $\red$ is large enough (equivalently, in case $\n+1\not\equiv0\pmod{\rho+1}$, $\rho\geq3$, and $\n$ is large enough).
	
	\begin{res}[{\cite[Theorem $6.3$]{DavydovEtAl1}}]\label{Res_NmultipleOfrhoplus1}
	    Let $t\in\NN$ and $\red=Rt+1$ (equivalently, $\n=(\rho+1)t$) with $R\geq4$ (equivalently, $\rho\geq3$).
	    Suppose that $q=(q')^R=(q')^{\rho+1}$, $q'\geq4$, and choose $t_0\in\NN$ such that
	    \[
	        q^{t_0-1}\geq(q'-1)\left(\frac{R(R+1)}{2}-2\right)+R+5\textnormal{.}
	    \]
	    If $t\geq t_0$, then there exists an $[\len,\len-\red]_qR$ code with
	    \[
	        \len=\left(\frac{R(R+1)}{2}-2\right)(q')^{\red-R}-\left(\frac{(R-1)R}{2}-7\right)(q')^{\red-R-1}\textnormal{, hence }\lenfunc(\red,R)\leq\len\textnormal{.}
	    \]
	    This code corresponds to a $\rho$-saturating set of $\pg(\n,q)$, with $\n=\red-1$ and $\rho=R-1$, of size
	    \[
	        \len=\left(\frac{(\rho+1)(\rho+2)}{2}-2\right)(q')^{\n-\rho}-\left(\frac{\rho(\rho+1)}{2}-7\right)(q')^{\n-\rho-1}\textnormal{, hence }\satbound(\n,\rho)\leq\len\textnormal{.}
	    \]
	\end{res}
	
	Whether we compare this result with Theorem \ref{Thm_MainUpperBoundSubgeometricRhoLargerThan1} or Corollary \ref{Crl_MainUpperBoundSubgeometricRLargerThan2}, one can easily see that our results do not improve Result \ref{Res_NmultipleOfrhoplus1} for given restrictions on $\red$, $R$ (equivalently, on $\n$, $\rho$), $t$ and $q$.
	Notice that we omitted the case $t=1$ from the original result.
	However, we discuss this particular case in Remark \ref{Rmk_ImprovementBy1}.
	
	\begin{res}[{\cite[Theorem $6.4$]{DavydovEtAl1}}]\label{Res_NmultipleOfGamma}
	    Let $t\in\NN$ and $\red=Rt+\gamma$ (equivalently, $\n+1=(\rho+1)t+\gamma$), $\gamma\in\{2,3,\dots,R-1\}$ with $R\geq4$ (equivalently, $\rho\geq3$).
	    Suppose that $q=(q')^R=(q')^{\rho+1}$ and choose $t_0\in\NN$ such that
	    \[
	        q^{t_0-1}\geq\len_{R,q}^{(\gamma)}\textnormal{, with }\len_{R,q}^{(\gamma)}:=\sum_{i=0}^\gamma(q'-1)^i\binom{R+\gamma}{i+1}\sim\binom{R+\gamma}{R-1}(q')^\gamma\textnormal{.}
	    \]
	    If $t=1$ or $t\geq t_0$, then there exists an $[\len,\len-\red]_qR$ code with
	    \[
	        \len=\len_{R,q}^{(\gamma)}\cdot (q')^{\red-R-\gamma}+w\frac{(q')^{\red-R-\gamma}-1}{q-1}\textnormal{, }0\leq w\leq R-3\textnormal{, hence }\lenfunc(\red,R)\leq\len\textnormal{.}
	    \]
	    This code corresponds to a $\rho$-saturating set of $\pg(\n,q)$, with $\n=\red-1$ and $\rho=R-1$, of size
	    \[
	        \len=\len_{\rho+1,q}^{(\gamma)}\cdot (q')^{\n-\rho-\gamma}+w\frac{(q')^{\n-\rho-\gamma}-1}{q-1}\textnormal{, }0\leq w\leq\rho-2\textnormal{, hence }\satbound(\n,\rho)\leq\len\textnormal{.}
	    \]
	\end{res}
	
	In general, our results (Theorem \ref{Thm_MainUpperBoundSubgeometricRhoLargerThan1} or Corollary \ref{Crl_MainUpperBoundSubgeometricRLargerThan2}) improves the result above; this improvement increases as $\gamma$ grows.
	Moreover, the authors of \cite{DavydovEtAl1} note that the main term of the asymptotic covering density of the infinite family of covering codes arising from Result \ref{Res_NmultipleOfGamma} is equal to $\left(\frac{(R+\gamma)^{R-1}}{(R-1)!}\right)^R\cdot\frac{1}{R!}$, which is significant larger than $\frac{\big(R(R-1)\big)^R}{R!}$ (see Corollary \ref{Crl_MainUpperBoundSubgeometricRLargerThan2}).
	
	\begin{rmk}\label{Rmk_Section7}
	    Davydov, Giulietti, Marcugini and Pambianco \cite[Section $7$]{DavydovEtAl1} cleverly extended the results above.
	    More specifically, if $R'$ is a proper divisor of $R$, $\red=Rt+s\frac{R}{R'}$ ($s\in\{1,2,\dots,R'-1\}$) and $q=(q')^{R'}$, they managed to construct infinite families covering codes of length roughly equal to $\frac{R}{R'}\binom{R'+s}{R'-1}q^\frac{\red-R}{R}$ if $t=1$ or $t$ is large enough (Note that Result \ref{Res_DavydovEtAlR2} exists under the same conditions, for $R'=2$).
	    Under these conditions, our results present improvement if $q$ is an $R^\textnormal{th}$ power, $R'\geq4$ and
	    \begin{enumerate}
	        \item either $t>1$ is relatively small, or
	        \item $\frac{R+1}{2}<\frac{1}{R'}\binom{R'+s}{R'-1}$.
	    \end{enumerate}
	    The latter condition is true if $R'$ is a relatively large divisor of $R$ ($R'\gtrsim s\sqrt[s]{R}$).
	    In all other cases, the results of \cite[Section $7$]{DavydovEtAl1} are better than ours.
	    
	    We would shortly want to point out a misprint in \cite[Corollary $7.5$]{DavydovEtAl1}.
	    It should state that $\gamma\in\{2,3,\dots,R'-1\}$ instead of $\gamma\in\{2,3,\dots,R-1\}$.
	\end{rmk}
	
	Lastly, Result \ref{Res_FourLinesQuadric} arose by cleverly choosing four disjoint lines in $\pg(3,q')$ of which the union of points forms a $3$-fold strong blocking set (and by subsequently $q^m$-concatenating the covering code arising from the obtained $2$-saturating set).
	This idea of choosing pairwise disjoint $(\n-\rho)$-spaces of $\pg(\n,q')$, of which the union of points forms a $(\rho+1)$-fold strong blocking set, led to the following, more general results if $q'$ is large enough.
	
	\begin{res}[{\cite[Theorem $24$]{FancsaliSziklai1} and \cite[Proposition $10$, Subsection $3.4$]{FancsaliSziklai2}}]\label{Res_HiggledyPiggledy}~
	    \begin{enumerate}
	        \item Let $q=(q')^\n$, $q'\geq2\n-1$.
	        Then
	        \[
	            \satbound(\n,\n-1)\leq(2\n-1)(q'+1)\textnormal{.}
	        \]
	        \item Let $1<\rho<\n$ and $q=(q')^{\rho+1}$, $q'>\n+1$.
	        Then
	        \[
	            \satbound(\n,\rho)\leq\big((\n-\rho+1)\rho+1\big)\frac{(q')^{\n-\rho+1}-1}{q'-1}\textnormal{.}
            \]
	    \end{enumerate}
	\end{res}
	
	If $q'>\n+1$, one can check that our results (Theorem \ref{Thm_MainUpperBoundSubgeometricRhoLargerThan1}) improve Result \ref{Res_HiggledyPiggledy}($2.$) if and only if $\rho<\frac{2\n-1}{3}$.
	We discuss the comparison between our results and Result \ref{Res_HiggledyPiggledy}($1.$) in Remark \ref{Rmk_ImprovementBy1}.
	
	\subsubsection*{The mixed subgeometry approach}
	
	The \emph{mixed subgeometry approach} is based on constructing saturating sets as a union of several distinct subgeometries which are not part of a common, larger subgeometry.
	This approach was the main source of inspiration for this article.
	The technique is used much less than the \emph{strong blocking set approach}.
	In fact, Result \ref{Res_DavydovMixedSubInPlane} below is the only instance using the \emph{mixed subgeometry approach} that we encountered in the literature.
	
	\begin{res}[{\cite[Theorem $5.2$]{Davydov1}}]\label{Res_DavydovMixedSubInPlane}
		Let $q$ be square.
		Let $b_1$, $b_2$ and $b_3$ be three distinct Baer sublines spanning $\pg(2,q)$ and sharing a common point $P$, with the addition that $b_1$ and $b_2$ share a further point $Q\neq P$ as well.
		Then $(b_1\cup b_2\cup b_3)\setminus\{P\}$ is a $1$-saturating set of $\pg(2,q)$.
		As a consequence,
		\[
		    \satbound(2,1) \leq 3\sqrt{q}-1\textnormal{.}
		\]
	\end{res}
	
	Better bounds on $\satbound(2,1)$, $q$ square, are known (see \cite[Proposition $9$]{DavydovEtAl3} for an overview, which includes improvements arising from Result \ref{Res_DavydovEtAlR2}).
	Interestingly, as noted in \cite[Remarks $3$ and $4$]{DavydovEtAl3}, if $q$ equals the square of a prime number, no better bound on $\satbound(2,1)$ than the one depicted in Result \ref{Res_DavydovMixedSubInPlane} is known in the literature.
	Our interest was mainly peaked by the underlying (sub)geometric construction.
	In fact, Construction \ref{Constr_Subgeometric} is basically a highly generalised version of the construction described in Result \ref{Res_DavydovMixedSubInPlane}.
	
	\bigskip
	By making use of variations of $q^m$-concatenating constructions, the following bound is obtained, generalising the bound of Result \ref{Res_DavydovMixedSubInPlane}.
	
	\begin{res}[{\cite[Example $6$, Equation $(33)$]{Davydov2}}]\label{Res_DavydovMixedSubInPlaneGENERAL}
	    Let $\n$ be even and $q\geq16$ be square.
	    Then
	    \[
	        \satbound(\n,1) \leq (3\sqrt{q}-1)q^{\frac{\n}{2}-1}+\left\lfloor q^{\frac{\n}{2}-2}\right\rfloor\textnormal{.}
	    \]
	\end{res}
	
	Theorem \ref{Thm_MainUpperBoundSubgeometric} implies, for $\n$ even and $q$ square, that
	\[
        \satbound(\n,1) \leq 3\sqrt{q}\cdot\left(q^{\frac{\n}{2}-1}+q^{\frac{\n}{2}-2}+\dots+1\right)-\frac{\n}{2}\textnormal{.}
    \]
    Hence, in case $\rho=1$, we only achieve improvement if $q\in\{4,9\}$.
	
	\begin{rmk}[The case $\rho=\n-1$]\label{Rmk_ImprovementBy1}
	    Observe the following:
	    \begin{itemize}
	        \item[(A)] One can illustrate the use of the \emph{strong blocking set approach} by considering three non-concurrent lines of a Baer subplane, proving that $\satbound(2,1) \leq 3\sqrt{q}$ if $q$ is square.
	        \item[(B)] On the other hand, the \emph{mixed subgeometry approach} led to Result \ref{Res_DavydovMixedSubInPlane}, which states that $\satbound(2,1) \leq 3\sqrt{q}-1$ if $q$ is square, hence improving the bound of (A) by $1$.
	    \end{itemize}
	    Curiously, if $q=(q')^\n$, we discover the same phenomenon when observing a specific generalisation of the construction behind each of these two bounds.
	    On the one hand, Davydov and \"Osterg\aa rd \cite[Theorem $6$]{DavydovOstergard} generalised (A) by constructing the so-called \emph{tetrahedron}, obtained by connecting $\n+1$ points of $\pg(\n,q')$ in general position.
	    This gives rise to the expression
	    \begin{equation}\label{Eq_RhoIsNminus1Strong}
	        \satbound(\n,\n-1) \leq \frac{\n(\n+1)}{2}q'-\frac{\n(\n-1)}{2}+1\textnormal{.}
	    \end{equation}
	    On the other hand, the construction behind Theorem \ref{Thm_MainUpperBoundSubgeometric} (\emph{mixed subgeometry approach}) generalises (B) and, if $\n>1$, gives rise to
	    \begin{equation}\label{Eq_RhoIsNminus1Subgeom}
	        \satbound(\n,\n-1) \leq \frac{\n(\n+1)}{2}q'-\frac{\n(\n-1)}{2}\textnormal{,}
	    \end{equation}
	    which improves \eqref{Eq_RhoIsNminus1Strong} yet again by $1$.
	    
	    Although this phenomenon is somewhat curious, there are better upper bounds known for $\rho=\n-1$, $q=(q')^\n$, $\n\geq3$.
	    If $q'>\n+1$, Result \ref{Res_HiggledyPiggledy}($2.$) generally improves \eqref{Eq_RhoIsNminus1Subgeom}; moreover, \cite[Corollary $3.12$]{DavydovEtAl1} states that
	    \[
	        \satbound(\n,\n-1) \leq \frac{\n(\n+1)}{2}q'-\frac{\n(\n-1)}{2}-2q'+7\textnormal{,}
	    \]
	    which clearly improves \eqref{Eq_RhoIsNminus1Subgeom} if $q'\geq4$.
	    In conclusion, our results only improve the case $\rho=\n-1$ if $\n\geq3$ and $q'\in\{2,3\}$.
	\end{rmk}
	
	\section{The geometries $\boldsymbol{Y(\rho,m,q')}$ and $\boldsymbol{T^*(\AffineHyperSub_{\rho,m,q'})}$}\label{Sect_PointLine}
	
	In this section, we put the topic of saturating sets temporarily on hold.
	We will focus on an isomorphism between a point-line geometry $Y(\rho,m,q')$ (see Definition \ref{Def_PointLineY}) and the \emph{linear representation} $T^*(\AffineHyperSub_{\rho,m,q'})$ (see Definition \ref{Def_LinearRep}) of a $\rho$-dimensional $q'$-subgeometry $\AffineHyperSub_{\rho,m,q'}$ of $\pg\big(\rho+1,(q')^m\big)$ ($m\in\NNnot$).
	This will be done by taking advantage of a coordinate system of the respective ambient projective spaces (although an alternative exists using field reduction, see Subsection \ref{Subsec_IsomorphismXFieldRed}).
	
	\bigskip
	We have reason to believe that the point-line geometry $Y(\rho,m,q')$ hasn't been considered in the literature before.\footnote{With exception of the case $m=1$, as there exists a well-known isomorphism between the affine parts of $q'$-sublines of $\pg\big(1,(q')^{\rho+1}\big)$ through a fixed point (which can be viewed as $q'$-sublines of $\ag\big(1,(q')^{\rho+1}\big)$) and the lines of $\ag(\rho+1,q')$.}
	
	\subsection{Preliminaries}
	
	We introduce two point-line geometries.
	Pay attention to the fact that within the first point-line geometry, points of $\pg\big(m,(q')^{\rho+1}\big)$ are considered, while in the second point-line geometry we consider points of $\pg\big(\rho+1,(q')^m\big)$.
	
	\begin{df}\label{Def_PointLineY}
	    Let $m\in\NNnot$.
	    Consider an $(m-1)$-dimensional $q'$-subgeometry $\HyperSub$ of $\pg\big(m,(q')^{\rho+1}\big)$ and define $\Sigma_\HyperSub:=\vspanqr{\HyperSub}$.
	    The point-line geometry $Y(\rho,m,q')$ is the incidence structure $(\ppointsSub,\llinesSub)$ with natural incidence, where
	    \begin{itemize}
	        \item $\ppointsSub$ is the set of points of $\pg\big(m,(q')^{\rho+1}\big)\setminus\Sigma_\HyperSub$, and
	        \item $\llinesSub$ is the set of all point sets $\Sub\setminus\HyperSub$, where $\Sub$ is an $m$-dimensional $q'$-subgeometry of $\pg\big(m,(q')^{\rho+1}\big)$ that contains $\HyperSub$.
	    \end{itemize}
	    If $\pg\big(m,(q')^{\rho+1}\big)$ is embedded in $\pg\big(m',(q')^{\rho+1}\big)$ ($m'\geq m$) as an $m$-dimensional subspace $\Pi$, we will use notation $\ppointsSub^\Pi$ and $\llinesSub^\Pi$, respectively, to avoid confusion when considering more than one such point-line geometry.
	\end{df}
	
	Secondly, we brush up the concept of \emph{linear representations}.
	This notion was independently introduced for hyperovals by Ahrens and Szekeres \cite{AhrensSzekeres} and Hall \cite{Hall}, and extended to general point sets by De Clerck \cite{DeClerck}.
	
	\begin{df}\label{Def_LinearRep}
	    Let $m\in\NNnot$.
	    Consider a point set $\mathcal{K}$ of $\pg\big(\rho,(q')^m\big)$ embedded in $\pg\big(\rho+1,(q')^m\big)$.
	    The \emph{linear representation} of $\mathcal{K}$ is the point-line geometry $T^*(\mathcal{K}):=(\mathscr{P}_\mathcal{K},\mathscr{L}_\mathcal{K})$ with natural incidence, where
	    \begin{itemize}
	        \item $\mathscr{P}_\mathcal{K}$ is the set of points of $\pg\big(\rho+1,(q')^m\big)\setminus\pg\big(\rho,(q')^m\big)$, and
	        \item $\mathscr{L}_\mathcal{K}$ is the set of all point sets $\ell\setminus\{P\}$, where $\ell\nsubseteq\pg\big(\rho,(q')^m\big)$ is a line of $\pg\big(\rho+1,(q')^m\big)$ intersecting $\mathcal{K}$ in a point $P$.
	    \end{itemize}
	\end{df}
	
	\subsection{A direct isomorphism between $\boldsymbol{Y(\rho,m,q')}$ and $\boldsymbol{T^*(\AffineHyperSub_{\rho,m,q'})}$ using coordinates}\label{Subsec_IsomorphismYCoordinates}
	
	Consider the following configuration.
	
	\begin{config}\label{Conf_BasicPointLineGeometryConstruction}
	    Let $m\in\NNnot$ and suppose $q=(q')^{\rho+1}$. Consider an $(m-1)$-dimensional $q'$-subgeometry $\HyperSub$ of $\pg(m,q)$ and define $\Sigma_\HyperSub:=\vspanq{\HyperSub}$.
	    Choose a coordinate system for $\pg(m,q)$ such that
	    \begin{itemize}
	        \item $E_1,\dots,E_m$ are the points with coordinates $(0,1,0,\dots,0),\dots,(0,0,0,\dots,1)$, respectively,
	        \item $E'$ is the point with coordinates $(0,1,\dots,1)$, and
	        \item $\HyperSub$ is the (by Lemma \ref{Lm_UniqueSubThroughFrame} unique) $(m-1)$-dimensional $q'$-subgeometry containing the points $E_1,\dots,E_m$ and $E'$.
	    \end{itemize}
	    Furthermore, let $\AffineHyperSub_{\rho,m,q'}$ be a $\rho$-dimensional $q'$-subgeometry of $\pg\big(\rho+1,(q')^m\big)$ and define $\Sigma_{\AffineHyperSub_{\rho,m,q'}}:=\vspanqm{\AffineHyperSub_{\rho,m,q'}}$.
	    Choose a coordinate system for $\pg\big(\rho+1,(q')^m\big)$ such that
	    \begin{itemize}
	        \item $\AffineHyperSub_{\rho,m,q'}$ is the (by Lemma \ref{Lm_UniqueSubThroughFrame} unique) $\rho$-dimensional $q'$-subgeometry containing all points corresponding to the set of coordinates $\{(0,1,0,\dots,0),\dots,(0,0,0,\dots,1),(0,1,1,\dots,1)\}$.
	    \end{itemize}
	    For notational simplicity, we will often write $\AffineHyperSub$ instead of $\AffineHyperSub_{\rho,m,q'}$.
	\end{config}
	
	\begin{lm}\label{Lm_CoordinatesOfPointsOfSub}
	    Consider Configuration \ref{Conf_BasicPointLineGeometryConstruction}.
	    Let $P,Q\notin\Sigma_\HyperSub$ be two distinct points of $\pg(m,q)$ with coordinates $(1,x_1,x_2,\dots,x_m)$ and $(1,y_1,y_2,\dots,y_m)$ ($x_i,y_i\in\FF_q$), such that $PQ$ intersects $\Sigma_\HyperSub$ in $E'$.
		Let $\Sub$ be the (by Lemma \ref{Lm_UniqueSubThroughHyperSubAndLine} unique) $m$-dimensional $q'$-subgeometry containing $\HyperSub$, $P$ and $Q$.
		Then the set of coordinates of all points in $\Sub\setminus\HyperSub$ is equal to
		\[
		    \left\{\big(1,x_1+k_1(y_1-x_1),x_2+k_2(y_1-x_1),\dots,x_m+k_m(y_1-x_1)\big):k_1,\dots,k_m\in\FF_{q'}\right\}\textnormal{.}
		\]
	\end{lm}
	\begin{proof}
	    It is clear that the hyperplane $\Sigma_\HyperSub$ is defined by the equation $X_0=0$.
	    Suppose that $E_0$ is the point of $\pg(m,q)$ with coordinates $(1,0,0,\dots,0)$ and $E$ the point with coordinates $(1,1,1,\dots,1)$, and let $\Sub_0$ be the (by Lemma \ref{Lm_UniqueSubThroughFrame} unique) $m$-dimensional $q'$-subgeometry containing the frame $\{E_0,\dots,E_m,E\}$. As this is the canonical frame,
	    it is clear that the set of coordinates of all points in $\Sub_0\setminus\Sigma_\HyperSub$ is equal to
	    \[
		    \left\{(1,k_1,k_2,\dots,k_m):k_1,\dots,k_m\in\FF_{q'}\right\}\textnormal{.}
		\]
	    One can find an element of $\textnormal{PGL}(m+1,q)$ that maps the canonical frame $\{E_0,E_1,\dots,E_m,E\}$ onto the frame $\{P,E_1,\dots,E_m,Q\}$, which can be represented by an $\FF_q$-multiple of the following matrix:
		\[
		    \begin{pmatrix}
                1 & 0 & 0 & \cdots & 0\\
                x_1 & y_1-x_1 & 0 & \cdots & 0\\
                x_2 & 0 & y_2-x_2 & \cdots & 0\\
                \vdots & \vdots & \vdots & \ddots & \vdots\\
                x_m & 0 & 0 & \cdots & y_m-x_m
            \end{pmatrix}\textnormal{.}
		\]
		Such a matrix maps a point of $\Sub_0$ with coordinates $(1,k_1,k_2,\dots,k_m)$, $k_i\in\FF_{q'}$, onto a point of $\Sub$ with coordinates
		\[
		    \big(1,x_1+k_1(y_1-x_1),x_2+k_2(y_2-x_2),\dots,x_m+k_m(y_m-x_m)\big)\textnormal{.}
		\]
		Note that, as $E'\in PQ$, the tuple $(0,y_1-x_1,y_2-x_2,\dots,y_m-x_m)$ has to be an $\FF_q$-multiple of $(0,1,1,\dots,1)$, which implies that $y_i-x_i=y_j-x_j$ for all $i,j\in\{1,2,\dots,m\}$.
		Hence, the set of coordinates of all points in $\Sub\setminus\HyperSub$ can be simplified to
		\[
		    \left\{\big(1,x_1+k_1(y_1-x_1),x_2+k_2(y_1-x_1),\dots,x_m+k_m(y_1-x_1)\big):k_1,\dots,k_m\in\FF_{q'}\right\}\textnormal{.}\qedhere
		\]
	\end{proof}
	
	We now introduce the following map $\varphi$.
	
	\begin{df}\label{Def_AffineIsomorphism}
	    Consider Configuration \ref{Conf_BasicPointLineGeometryConstruction}.
	    Choose elements $\alpha\in\FF_q$ and $\beta\in\FF_{(q')^m}$ such that $\FF_{q'}[\alpha]\cong\FF_q$ and $\FF_{q'}[\beta]\cong\FF_{(q')^m}$.
		Define the map $\varphi:\ppointsSub\rightarrow\ppointsAff$ that maps a point of $\ppointsSub$ with coordinates
		\[
		    (1,z_1,z_2,\dots,z_m) = \left(1,\sum_{j=0}^\rho z_{1j}\alpha^j,\sum_{j=0}^\rho z_{2j}\alpha^j,\dots,\sum_{j=0}^\rho z_{mj}\alpha^j\right) \qquad (z_k\in\FF_q,z_{rs}\in\FF_{q'})
		\]
		onto the unique point of $\ppointsAff$ with coordinates
		\[
		    \left(1,\sum_{i=1}^m z_{i0}\beta^{i-1},\sum_{i=1}^m z_{i1}\beta^{i-1},\dots,\sum_{i=1}^m z_{i\rho}\beta^{i-1}\right)\textnormal{.}
		\]
		If $\pg(m,q)$ is embedded in a larger projective geometry as a subspace $\Pi$, we will use the notation $\restr{\varphi}{\Pi}$ to clarify which map is considered.
	\end{df}
	
	\begin{thm}\label{Thm_SubgeometriesAreAffineLines}
		Let $m\in\NNnot$.
		Then the point-line geometries $Y(\rho,m,q')$ and $T^*(\AffineHyperSub_{\rho,m,q'})$ are isomorphic.
	\end{thm}
	\begin{proof}
	    Let $q=(q')^{\rho+1}$.
	    Note that the choice of coordinates made in Configuration \ref{Conf_BasicPointLineGeometryConstruction} does not affect the generality of the theorem.
	    After all, any collineation of $\pg(m,q)$ preserves elements of $\llinesSub$ as being (the affine parts of) $m$-dimensional $q'$-subgeometries containing the image of $\HyperSub$, hence the whole set $\llinesSub$ is preserved and, furthermore, incidence is sustained.
	    The same holds for the point-line geometry $T^*(\AffineHyperSub_{\rho,m,q'})$.
	    
		Thus, consider Configuration \ref{Conf_BasicPointLineGeometryConstruction}.
		Let $\Sub\setminus\HyperSub$ be an arbitrary element of $\llinesSub$.
		Suppose that $P,Q\in\Sub\setminus\HyperSub$ are two distinct points with coordinates $(1,x_1,x_2,\dots,x_m)$ and $(1,y_1,y_2,\dots,y_m)$ ($x_i,y_i\in\FF_q$), such that $PQ$ intersects $\Sigma_\HyperSub$ in $E'$.
		By Lemma \ref{Lm_UniqueSubThroughHyperSubAndLine}, $\Sub$ is uniquely defined by $\HyperSub$, $P$ and $Q$.
		By Lemma \ref{Lm_CoordinatesOfPointsOfSub}, the set of coordinates of all points in $\Sub\setminus\HyperSub$ is equal to
		\[
		    \left\{\big(1,x_1+k_1(y_1-x_1),x_2+k_2(y_1-x_1),\dots,x_m+k_m(y_1-x_1)\big):k_1,\dots,k_m\in\FF_{q'}\right\}\textnormal{.}
		\]
		Consider the map $\varphi$ (see Definition \ref{Def_AffineIsomorphism}); note that $\varphi$ is a bijection, as one can easily define its inverse.
		We will prove that $\varphi$ induces an isomorphism between the point-line geometries $Y(\rho,m,q')$ and $T^*(\AffineHyperSub_{\rho,m,q'})$.
		
		If $x_i=\sum_{j=0}^\rho x_{ij}\alpha^j$ and $y_i=\sum_{j=0}^\rho y_{ij}\alpha^j$ ($x_{ij},y_{ij}\in\FF_{q'}$), then the set of coordinates of the images of all points in $\Sub\setminus\HyperSub$ under $\varphi$ is equal to
		\begin{align}
		    &\left\{\left(1,\sum_{i=1}^m\big(x_{i0}+k_i(y_{10}-x_{10})\big)\beta^{i-1},\dots,\sum_{i=1}^m\big(x_{i\rho}+k_i(y_{1\rho}-x_{1\rho})\big)\beta^{i-1}\right):k_1,\dots,k_m\in\FF_{q'}\right\}\nonumber\\
		    =&\Bigg\{\left(1,\sum_{i=1}^mx_{i0}\beta^{i-1},\dots,\sum_{i=1}^mx_{i\rho}\beta^{i-1}\right)\nonumber\\
		    &\qquad\qquad\qquad+\sum_{i=1}^mk_i\beta^{i-1}\cdot\left(0,y_{10}-x_{10},\dots,y_{1\rho}-x_{1\rho}\right):k_1,\dots,k_m\in\FF_{q'}\Bigg\}\nonumber\\
		    =&\left\{\left(1,\sum_{i=1}^mx_{i0}\beta^{i-1},\dots,\sum_{i=1}^mx_{i\rho}\beta^{i-1}\right)+k\cdot\left(0,y_{10}-x_{10},\dots,y_{1\rho}-x_{1\rho}\right):k\in\FF_{(q')^m}\right\}\textnormal{.}\label{Eq_Coordinates}
		\end{align}
		The latter set is equal to the set of coordinates of all points on $l\setminus\Sigma_{\AffineHyperSub}$, with $l$ a line of $\pg\big(\rho+1,(q')^m\big)$ through $\varphi(P)\notin\Sigma_\AffineHyperSub$ intersecting $\Sigma_\AffineHyperSub$ in the point of $\AffineHyperSub$ with coordinates $(0,x_{10}-y_{10},\dots,x_{1\rho}-y_{1\rho})\subseteq\FF_{q'}^{\rho+2}$.
		Hence, as it is clear that $\varphi$ maps points on a line of $Y(\rho,m,q')$ onto points on a line of $T^*(\AffineHyperSub)$, this map naturally induces a morphism from $Y(\rho,m,q')$ to $T^*(\AffineHyperSub)$.
		
		As $\varphi$ is a bijection between $\ppointsSub$ and $\ppointsAff$, this map is injective w.r.t.\ the line sets $\llinesSub$ and $\llinesAff$, hence the only thing left to prove is the fact that $\varphi$ is surjective w.r.t.\ these line sets.
		
		It is clear that an element $l\setminus\{D'\}$ of $\llinesAff$ is uniquely defined by the point $D'\in\AffineHyperSub$ and a point $P'\in l\setminus\{D'\}$.
		By observing \eqref{Eq_Coordinates}, it is clear that we can fix the point $\overline{P}:=\varphi^{-1}(P')\in\pg(m,q)\setminus\Sigma_\HyperSub$ which has, let us say, coordinates $(1,\overline{x}_1,\dots,\overline{x}_m)\in\FF_q^{m+1}$, and try to choose a point $\overline{Q}\in\pg(m,q)\setminus\Sigma_\HyperSub$ such that
		\begin{enumerate}
		    \item $\overline{P}\overline{Q}$ intersects $\Sigma_\HyperSub$ in a point of $\HyperSub$, and
		    \item the element of $\llinesSub$ defined by $\HyperSub$, $\overline{P}$ and $\overline{Q}$ is mapped by $\varphi$ onto a line intersecting $\Sigma_\AffineHyperSub$ in $D'\in\AffineHyperSub$.
		\end{enumerate}
		Condition $1$ is clearly fulfilled if we choose a point $\overline{Q}\in\pg(m,q)$ with coordinates $(1,\overline{y}_1,\dots,\overline{y}_m)$ such that $\overline{y}_1-\overline{x}_1=\overline{y}_2-\overline{x}_2=\dots=\overline{y}_m-\overline{x}_m$.
		These equations mean that, once we fix the value $\overline{y}_1\in\FF_q$, we fix the entire point $\overline{Q}$.
		This implies that we have freedom to choose any value $\overline{y}_1=\sum_{j=0}^\rho \overline{y}_{1j}\alpha^j\in\FF_q$ to try and satisfy condition $2$.
		By observing \eqref{Eq_Coordinates}, it is clear that this freedom of choice implies that we can reach each tuple of coordinates in $\{0\}\times\FF_{q'}^{\rho+1}$ corresponding to a point of $\AffineHyperSub$, in particular the coordinates of $D'\in\AffineHyperSub$.
	\end{proof}
	
	\subsection{Notions of parallelism and independence in $\boldsymbol{Y(\rho,m,q')}$}\label{Subsec_ParallelismIndependence}
	
	Theorem \ref{Thm_SubgeometriesAreAffineLines} states that the point-line geometry $Y(\rho,m,q')$ is isomorphic to the linear representation $T^*(\AffineHyperSub_{\rho,m,q'})$ of a $\rho$-dimensional $q'$-subgeometry $\AffineHyperSub_{\rho,m,q'}$ of $\pg\big(\rho,(q')^m\big)$ (see Definition \ref{Def_PointLineY} and Definition \ref{Def_LinearRep}).
	As the lines of a linear representation are embedded in an affine geometry, notions of \emph{parallelism} and \emph{independence of concurrent lines} seem transferable to the point-line geometry $Y(\rho,m,q')$.
	
	\begin{df}\label{Def_Notions}
	    Let $m\in\NNnot$ and consider the point-line geometry $Y(\rho,m,q')$.
	    Then
	    \begin{itemize}
	        \item distinct lines of $\llinesSub$ are called \emph{concurrent} if they contain a common point.
	    \end{itemize}
	    Now define $\widetilde{\varphi}:=\varphi\circ\varphi_Y$, where $\varphi_Y$ is a collineation of $\pg(m,q)$ such that the image of the point-line geometry $Y(\rho,m,q')$ under $\varphi_Y$ corresponds to the coordinate system described in Configuration \ref{Conf_BasicPointLineGeometryConstruction}, and $\varphi$ is the isomorphism described in Definition \ref{Def_AffineIsomorphism}.
	    Then
	    \begin{itemize}
	        \item two lines of $\llinesSub$ are said to be $\widetilde{\varphi}$\emph{-parallel} if their images under $\widetilde{\varphi}$ are parallel,
	        \item concurrent lines of $\llinesSub$ are said to be $\widetilde{\varphi}$\emph{-independent} if their images under $\widetilde{\varphi}$ are independent\footnote{Concurrent lines $\ell_1,\ell_2,\dots,\ell_d$ of a projective space are called \emph{independent} if they span a subspace of dimension $d$.},
	        \item we will call (the point set of) any line of $\llinesSub$ a $1$\emph{-dimensional affine} $\widetilde{\varphi}$\emph{-subspace}.
	        Recursively, for any $d\in\{2,\dots,\rho+1\}$, a union of $(q')^{(d-1)m}$ $\widetilde{\varphi}$-parallel lines of $\llinesSub$ is said to be a $d$\emph{-dimensional affine} $\widetilde{\varphi}$\emph{-subspace} if each of these lines intersects a fixed $(d-1)$-dimensional affine $\widetilde{\varphi}$-subspace in precisely one point.
	    \end{itemize}
	\end{df}
	
	\begin{lm}\label{Lm_Invariance}
	    The notions described in Definition \ref{Def_Notions} are well-defined.
	\end{lm}
	\begin{proof}
	    Assume that $Y(\rho,m,q')$ is chosen in such a way that it corresponds to the coordinate system described in Configuration \ref{Conf_BasicPointLineGeometryConstruction}.
	    For this lemma to be true, we want to prove that if one of the last three notions described in Definition \ref{Def_Notions} holds for certain points or lines of $Y(\rho,m,q')$, it still holds for the images of these points or lines w.r.t.\ any collineation of the ambient geometry $\pg(m,q)$.
	    
	    Let $\Sub\setminus\HyperSub$ be an arbitrary element of $\llinesSub$ and $P,Q\in\Sub\setminus\HyperSub$ be two distinct points with coordinates $(1,x_1,x_2,\dots,x_m)$ and $(1,y_1,y_2,\dots,y_m)$, respectively ($x_i,y_i\in\FF_q$), such that $PQ$ intersects $\Sigma_\HyperSub$ in $E'$ (see Configuration \ref{Conf_BasicPointLineGeometryConstruction}), implying that
		\begin{equation}\label{Eq_FirstCoord}
		    y_1-x_1=\dots=y_m-x_m\textnormal{.}
		\end{equation}
		Assuming that $x_i=\sum_{j=0}^\rho x_{ij}\alpha^j$ and $y_i=\sum_{j=0}^\rho y_{ij}\alpha^j$ ($\FF_{q'}[\alpha]\cong\FF_q$ and $x_{ij},y_{ij}\in\FF_{q'}$, see Definition \ref{Def_AffineIsomorphism}), just as in \eqref{Eq_Coordinates}, the set of coordinates of the images of all points in $\Sub\setminus\HyperSub$ under $\varphi$ is equal to
		\[
		    \left\{\left(\textnormal{coordinates of }\varphi(P)\right)+k\cdot\left(0,y_{10}-x_{10},\dots,y_{1\rho}-x_{1\rho}\right):k\in\FF_{(q')^m}\right\}\textnormal{.}
		\]
		For any element $\overline{\Sub}\setminus\HyperSub$ of $\llinesSub$, we can analogously choose two distinct points $\overline{P},\overline{Q}\in\overline{\Sub}\setminus\HyperSub$ with coordinates $(1,\overline{x_1},\overline{x_2},\dots,\overline{x_m})$ and $(1,\overline{y_1},\overline{y_2},\dots,\overline{y_m})$, respectively ($\overline{x_i},\overline{y_i}\in\FF_q$), such that $\overline{P}\overline{Q}$ intersects $\Sigma_\HyperSub$ in $E'$, implying that
		\begin{equation}\label{Eq_SecondCoord}
		    \overline{y_1}-\overline{x_1}=\dots=\overline{y_m}-\overline{x_m}\textnormal{.}
		\end{equation}
		The set of coordinates of the images of all points in $\overline{\Sub}\setminus\HyperSub$ under $\varphi$ is equal to
		\[
		    \left\{\left(\textnormal{coordinates of }\varphi(\overline{P})\right)+k\cdot\left(0,\overline{y_{10}}-\overline{x_{10}},\dots,\overline{y_{1\rho}}-\overline{x_{1\rho}}\right):k\in\FF_{(q')^m}\right\}\textnormal{.}
		\]
		Hence, $\Sub\setminus\HyperSub$ and $\overline{\Sub}\setminus\HyperSub$ are $\varphi$-parallel if and only if there exists an element $\gamma\in\FF_{(q')^m}$ such that $(0,y_{10}-x_{10},\dots,y_{1\rho}-x_{1\rho})=\gamma\cdot(0,\overline{y_{10}}-\overline{x_{10}},\dots,\overline{y_{1\rho}}-\overline{x_{1\rho}})$.
		Moreover, as both of these coordinates are tuples of $\FF_{q'}^{\rho+2}$, $\gamma$ has to be an element of $\FF_{q'}$.
		Combining this with \eqref{Eq_FirstCoord} and \eqref{Eq_SecondCoord}, we get that $y_{ij}-x_{ij}=\gamma(\overline{y_{ij}}-\overline{x_{ij}})$ for a $\gamma\in\FF_{q'}$ independent of $i$ or $j$, implying that
		\begin{equation}\label{Eq_Parallelism}
		    (y_1-x_1,\dots,y_m-x_m)=\gamma\cdot(\overline{y_1}-\overline{x_1},\dots,\overline{y_m}-\overline{x_m})\quad\textnormal{for a }\gamma\in\FF_{q'}\textnormal{.}
		\end{equation}
		In conclusion, two elements $\Sub\setminus\HyperSub$ and $\overline{\Sub}\setminus\HyperSub$ are $\varphi$-parallel if and only if there exist two points $P$ and $Q$ in $\Sub\setminus\HyperSub$ on a line through $E'$ and two points $\overline{P}$ and $\overline{Q}$ in $\overline{\Sub}\setminus\HyperSub$ on a line through $E'$ such that property \eqref{Eq_Parallelism} holds for their respective coordinates.
		As $\FF_{q'}$ is fixed under any automorphism of the ambient field $\FF_q$, we can observe that property \eqref{Eq_Parallelism} stays valid if $P$, $Q$, $\overline{P}$ and $\overline{Q}$ are moved by any collineation of $\pg(m,q)$.
		This implies that $\varphi$-parallelism is invariant w.r.t.\ collineations, hence this notion is well-defined.
		
		Using this, we can prove the same for the notion of a $d$-dimensional affine $\widetilde{\varphi}$-subspace ($d\in\{1,\dots,\rho+1\}$).
		If $d=1$, this is trivially true.
		Note that any such a $d$-dimensional affine $\widetilde{\varphi}$-subspace is a set of $(q')^{dm}$ points of $\ppointsSub$.
		If $d\geq2$, such a $d$-dimensional affine $\widetilde{\varphi}$-subspace occurs as a union of $\widetilde{\varphi}$-parallel lines through each of the points of a $(d-1)$-dimensional affine $\widetilde{\varphi}$-subspace.
		As any collineation of $\pg(m,q)$ preserves incidence, $\widetilde{\varphi}$-parallelism and, inductively, $(d-1)$-dimensional affine $\widetilde{\varphi}$-subspaces, the proof follows.
		
		Finally, as the points of $d$ concurrent, $\widetilde{\varphi}$-independent lines of $\llinesSub$ are contained in a unique $d$-dimensional affine $\widetilde{\varphi}$-subspace, the invariance of the latter affine $\widetilde{\varphi}$-subspace implies the invariance of the $\widetilde{\varphi}$-independence of those lines.
	\end{proof}
	
	Consider the following configuration.
	
	\begin{config}\label{Conf_BasicThreeSubspaceConstruction}
		Let $m\in\NNnot$ and suppose $q=(q')^{\rho+1}$.
		Consider an $(m-2)$-dimensional $q'$-subgeometry $\HyperSub$ of $\pg(m,q)$ and define $\Sigma_\HyperSub:=\vspanq{\HyperSub}$.
		Let $\Pi_1$, $\Pi_2$ and $\Pi_3$ be three distinct hyperplanes of $\pg(m,q)$ through $\Sigma_\HyperSub$.
	\end{config}
	
	\begin{lm}\label{Lm_UniqueBigSubOverConstruction}
	    Consider Configuration \ref{Conf_BasicThreeSubspaceConstruction}.
	    Let $\Sub\supseteq\HyperSub$ be an $(m-1)$-dimensional $q'$-subgeometry of $\pg(m,q)$ such that $\vspanq{\Sub}=\Pi_1$ and consider a point $S\in\Pi_2\setminus\Sigma_\HyperSub$.
	    Then there exists a unique $m$-dimensional $q'$-subgeometry $\BigSub$ containing $\Sub$ and $S$ and intersecting $\Pi_3\setminus\Sigma_\HyperSub$.
	\end{lm}
	\begin{proof}
	    Take a point $R\in\Sub\setminus\HyperSub$.
	    Then the line $RS$ has to intersect $\Pi_3$ in a point $T\notin\Sigma_\HyperSub$.
	    Any $m$-dimensional $q'$-subgeometry that contains $\Sub$ and $S$ and intersects $\Pi_3$ in a point outside of $\Sigma_\HyperSub$, has to contain $T$ and, in particular, the (by Lemma \ref{Lm_UniqueSubThroughFrame}) unique $q'$-subline defined by $R$, $S$ and $T$.
	    By Lemma \ref{Lm_UniqueSubThroughHyperSubAndLine}, there exists exactly one such $q'$-subgeometry.
	\end{proof}
	
	\begin{df}
	    Consider Configuration \ref{Conf_BasicThreeSubspaceConstruction}.
	    Then we can define, for any point $S\in\ppointsSub^{\Pi_2}$, the \emph{projection map}
	    \[
	        \pr{\Pi_1}{\Pi_3}{S}:\llinesSub^{\Pi_1}\rightarrow\llinesSub^{\Pi_3}:\Sub\setminus\HyperSub\mapsto\left(\BigSub\cap\Pi_3\right)\setminus\HyperSub\textnormal{,}
	    \]
	    and the \emph{shadow map}
	    \[
	        \sh{\Pi_1}{\Pi_3}{S}:\llinesSub^{\Pi_1}\rightarrow\llinesSub^{\Pi_2}:\Sub\setminus\HyperSub\mapsto\left(\BigSub\cap\Pi_2\right)\setminus\HyperSub\textnormal{,}
	    \]
	    with $\BigSub$ the (by Lemma \ref{Lm_UniqueBigSubOverConstruction}) unique $m$-dimensional $q'$-subgeometry containing $\Sub$ and $S$ and intersecting $\Pi_3\setminus\Sigma_\HyperSub$.
	    Furthermore, for a fixed element $\Sub\setminus\HyperSub\in\llinesSub^{\Pi_1}$, we can naturally extend the definition above and define, for any subset $\mathcal{T}\subseteq\ppointsSub^{\Pi_2}$,
	    \[
	        \prArg{\Pi_1}{\Pi_3}{\mathcal{T}}{\Sub\setminus\HyperSub}:=\bigcup_{S\in\mathcal{T}}\prArg{\Pi_1}{\Pi_3}{S}{\Sub\setminus\HyperSub}
	    \]
	    and
	    \[
	        \shArg{\Pi_1}{\Pi_3}{\mathcal{T}}{\Sub\setminus\HyperSub}:=\bigcup_{S\in\mathcal{T}}\shArg{\Pi_1}{\Pi_3}{S}{\Sub\setminus\HyperSub}\textnormal{.}
	    \]
	\end{df}
	
	\begin{lm}\label{Lm_ProjectionAndShadowStayParallel}
	    Consider Configuration \ref{Conf_BasicThreeSubspaceConstruction}.
	    Let $\Sub\setminus\HyperSub\in\llinesSub^{\Pi_1}$ and $S_1,S_2\in\ppointsSub^{\Pi_2}$.
	    Then $\shArg{\Pi_1}{\Pi_3}{S_1}{\Sub\setminus\HyperSub}$ and $\shArg{\Pi_1}{\Pi_3}{S_2}{\Sub\setminus\HyperSub}$ are $\restr{\widetilde{\varphi}}{\Pi_2}$-parallel.
	\end{lm}
	\begin{proof}
	    Choose coordinates for $\pg(m,q)$ and consider the points
	    \[
	        E_0(1,0,0,\dots,0), E_1(0,1,0,\dots,0),\dots, E_m(0,0,\dots,0,1)
	    \]
	    and
	    \[
	         E(1,1,1,\dots,1), E'(0,1,1,\dots,1) \textnormal{ and } E''(0,0,1,\dots,1)\textnormal{.}
	    \]
	    By Lemma \ref{Lm_Invariance} we can assume, without loss of generality, that
	    \begin{itemize}
	        \item $\HyperSub$ is (by Lemma \ref{Lm_UniqueSubThroughFrame}) uniquely defined by the points $E_2,\dots,E_m$ and $E''$,
	        \item $E_1$ is a point of $\Sub$, and
	        \item $E_0$ and $E$ are the points $l\cap\Pi_2$ and $l\cap\Pi_3$, respectively, with $l\nsubseteq\Pi_1$ an arbitrarily chosen line intersecting $\Pi_1$ in a point of $\Sub\setminus\left(\HyperSub\cup\{E_1\}\right)$.
	    \end{itemize}
	    In this way, $\Sub$ is (indirectly by Lemma \ref{Lm_UniqueSubThroughHyperSubAndLine}) uniquely defined by $\HyperSub$, $E_1$ and $E'\in E_0E=l$.
	    Assuming $S_1$ has coordinates $(1,0,x_2,\dots,x_m)$ ($x_i\in\FF_q$), the line $E'S_1$ intersects $\Pi_3$ in a point $T_1$ with coordinates $(1,1,1+x_2,\dots,1+x_m)$.
	    
	    By Lemma \ref{Lm_UniqueBigSubOverConstruction}, there exists a unique $m$-dimensional $q'$-subgeometry $\BigSub$ containing $\Sub$, $S_1$ and $T_1$.
	    By Lemma \ref{Lm_CoordinatesOfPointsOfSub}, the set of coordinates of all points in $\BigSub\setminus\Sub$ is equal to
	    \[
		    \left\{\big(1,k_1,x_2+k_2,\dots,x_m+k_m\big):k_1,\dots,k_m\in\FF_{q'}\right\}\textnormal{.}
		\]
	    As a consequence, the set of coordinates of all points of $\shArg{\Pi_1}{\Pi_3}{S_1}{\Sub\setminus\HyperSub}$ is equal to
	    \begin{equation}\label{Eq_ShadowCoordinates}
		    \left\{\big(1,0,x_2+k_2,\dots,x_m+k_m\big):k_2,\dots,k_m\in\FF_{q'}\right\}\textnormal{.}
		\end{equation}
		Restricting these coordinates to the geometry $\pg(m-1,q)\cong\Pi_2$ (hence by ignoring the second coordinate $0$), the set of coordinates of the images of all points \eqref{Eq_ShadowCoordinates} under $\restr{\varphi}{\Pi_2}$ is, as in \eqref{Eq_Coordinates}, equal to
		\[
		    \left\{\left(\textnormal{coordinates of }\restr{\varphi}{\Pi_2}(S_1)\right)+k\cdot\left(0,1,0,\dots,0\right):k\in\FF_{(q')^{m-1}}\right\}\textnormal{.}
		\]
		As the line parallel class of the affine line that arises in this way does not rely on the choice of the point $S_1\in\ppointsSub^{\Pi_2}$, the lemma is proven.
	\end{proof}
	
	\begin{lm}\label{Lm_ProjectionAndShadowGetBigger}
	    Consider Configuration \ref{Conf_BasicThreeSubspaceConstruction}.
	    Let $\Sub_1\setminus\HyperSub,\Sub_2\setminus\HyperSub,\dots,\Sub_j\setminus\HyperSub$ be $j$ distinct, $\restr{\widetilde{\varphi}}{\Pi_1}$-independent elements of $\llinesSub^{\Pi_1}$ sharing a point $F\in\ppointsSub^{\Pi_1}$ $(j\in\{1,2,\dots,\rho\})$ and suppose $\mathcal{T}\subseteq\ppointsSub^{\Pi_2}$ is a $(j-1)$-dimensional affine $\restr{\widetilde{\varphi}}{\Pi_2}$-subspace.
	    Then there exists a $k\in\{1,2,\dots,j\}$ such that $\prArg{\Pi_1}{\Pi_3}{\mathcal{T}}{\Sub_k\setminus\HyperSub}$ is a $j$-dimensional affine $\restr{\widetilde{\varphi}}{\Pi_2}$-subspace.
	\end{lm}
	\begin{proof}
	    Choose a point $F'\in\mathcal{T}$ and define $F'':=FF'\cap\Pi_3$.
	    The projection of points of $\Pi_1$ onto $\Pi_2$ through the point $F''$ is a natural projectivity between the spaces when interpreted as distinct projective geometries.
	    Hence, if one projects $\Sub_1\setminus\HyperSub,\Sub_2\setminus\HyperSub,\dots,\Sub_j\setminus\HyperSub$ onto $\Pi_2$ in this way, we obtain $j$ distinct, $\restr{\widetilde{\varphi}}{\Pi_2}$-independent elements $\Sub_1'\setminus\HyperSub,\Sub_2'\setminus\HyperSub,\dots,\Sub_j'\setminus\HyperSub$ of $\llinesSub^{\Pi_2}$ sharing the point $F'\in\mathcal{T}$.
	    As $\mathcal{T}$ is a $(j-1)$-dimensional affine $\restr{\widetilde{\varphi}}{\Pi_2}$-subspace, there has to exist a $\Sub_k'\setminus\HyperSub\in\llinesSub^{\Pi_2}$ which has only the point $F'$ in common with $\mathcal{T}$.
	    Moreover, it is easy to see that $\Sub_k'\setminus\HyperSub=\shArg{\Pi_1}{\Pi_3}{F'}{\Sub_k\setminus\HyperSub}$.
	    Hence, by Lemma \ref{Lm_ProjectionAndShadowStayParallel}, $\shArg{\Pi_1}{\Pi_3}{\mathcal{T}}{\Sub_k\setminus\HyperSub}$ is a union of $|\mathcal{T}|$ distinct, $\restr{\widetilde{\varphi}}{\Pi_2}$-parallel elements of $\llinesSub^{\Pi_2}$, each containing a unique point of $\mathcal{T}$.
	    In other words, $\shArg{\Pi_1}{\Pi_3}{\mathcal{T}}{\Sub_k\setminus\HyperSub}$ is a $j$-dimensional affine $\restr{\widetilde{\varphi}}{\Pi_2}$-subspace.
	    By considering the natural projection of points of $\Pi_2$ onto $\Pi_3$ through $F$, one can easily check that $\shArg{\Pi_1}{\Pi_3}{\mathcal{T}}{\Sub_k\setminus\HyperSub}$ gets projected onto $\prArg{\Pi_1}{\Pi_3}{\mathcal{T}}{\Sub_k\setminus\HyperSub}$.
	\end{proof}
	
	\subsection{A detour $\boldsymbol{Y(\rho,m,q')\sim X(\rho,m,q')\sim T^*(\AffineHyperSub_{\rho,m,q'})}$ using field reduction}\label{Subsec_IsomorphismXFieldRed}
	
	In Subsection \ref{Subsec_IsomorphismYCoordinates}, we successfully obtained an isomorphism between $Y(\rho,m,q')$ and $T^*(\AffineHyperSub_{\rho,m,q'})$ using coordinates.
    The notation $Y(\rho,m,q')$, however, was chosen for a reason, as we want to emphasize its link with the point-line geometry $X(\rho,m,q')$ introduced by De Winter, Rottey and Van de Voorde.
    
    \begin{df}[{\cite[Section $3$]{DeWinterRotteyVandeVoorde}}]\label{Def_PointLineX}
	    Let $m\in\NNnot$.
	    Consider a $\rho$-dimensional subspace $\pi$ of $\pg(\rho+m,q')$.
	    The point-line geometry $X(\rho,m,q')$ is the incidence structure $(\ppointsX,\llinesX)$ with natural incidence, where
	    \begin{itemize}
	        \item $\ppointsX$ is the set of all $(m-1)$-spaces of $\pg(\rho+m,q')$ disjoint to $\pi$, and
	        \item $\llinesX$ is the set of all $m$-spaces of $\pg(\rho+m,q')$ meeting $\pi$ exactly in one point.
	    \end{itemize}
	\end{df}
    
    In their work, the authors construct an explicit isomorphism between $X(\rho,m,q')$ and $T^*(\AffineHyperSub_{\rho,m,q'})$.
    
    \begin{thm}[{\cite[Theorem $4.1$]{DeWinterRotteyVandeVoorde}}]\label{Thm_DeWinterIsomorphism}
	    Let $m\in\NNnot$.
	    Let $\AffineHyperSub_{\rho,m,q'}$ be a $\rho$-dimensional $q'$-subgeometry of $\pg\big(\rho,(q')^m\big)$.
	    Then the point-line geometries $X(\rho,m,q')$ and $T^*(\AffineHyperSub_{\rho,m,q'})$ are isomorphic.
	\end{thm}
    
    This theorem, together with Theorem \ref{Thm_SubgeometriesAreAffineLines}, implies that $Y(\rho,m,q')\sim X(\rho,m,q')$.
    
    \bigskip
    As a side note, we would like to mention that it is possible to work the other way around by constructing a direct isomorphism between $Y(\rho,m,q')$ and $X(\rho,m,q')$ (different from the composition of the isomorphisms behind Theorem \ref{Thm_SubgeometriesAreAffineLines} and Theorem \ref{Thm_DeWinterIsomorphism}) by using \emph{field reduction} (see \cite{LavrauwVandeVoorde} for an introduction to this technique).
    Although this results in a more elegant isomorphism between these two point-line geometries, we didn't include it in this work as we needed an explicit isomorphism describing the link between $Y(\rho,m,q')$ and $T^*(\AffineHyperSub_{\rho,m,q'})$ to exploit notions of parallelism and independence of concurrent lines (see Subsection \ref{Subsec_ParallelismIndependence}).
    Besides, the description and proof of this isomorphism easily takes up several pages.
	
	\section{Constructing saturating sets using subgeometries}\label{Sect_Subgeometric}
	
	In this section, we switch our focus back to saturating sets.
	We will construct a point set (see Construction \ref{Constr_Subgeometric}) in $\pg(\n,q)$, $q=(q')^{\rho+1}$, and prove that this is a $\rho$-saturating set.
	The existence of this saturating set directly leads to the main result of this article, namely that $\satbound(\n,\rho)\lesssim\frac{(\rho+1)(\rho+2)}{2}(q')^{\n-\rho}$ (see Theorems \ref{Thm_MainUpperBoundSubgeometric} and \ref{Thm_MainUpperBoundSubgeometricRhoLargerThan1}).
	
	\subsection{Preliminaries}
	
	We will make use of the following concept.
	
	\begin{df}\label{Def_Flower}
        Let $m\in\NNnot$.
        Let $s\in\NN$, let $t\in\NN\setminus\{0,1\}$ and consider an $(s-1)$-subspace $\Sigma$ in $\pg(m,q)$.
        A set of independent\footnote{Alternatively, one can state that $\dim{\tau_1\cap\dots\cap\tau_t}=s-1$ and $\dim{\vspan{\tau_1,\dots,\tau_t}}=s+t-1$.} $s$-subspaces $\Flower:=\{\tau_1,\dots,\tau_t\}$ through $\Sigma$ is called an \emph{$s$-flower with pistil $\Sigma$}.
        The elements of $\Flower$ are called the \emph{petals} of the $s$-flower.
    \end{df}
    
    Furthermore, we will at one point need the following property concerning $q'$-subgeometries.
	
	\begin{lm}\label{Lm_UniqueSubThroughHyperSubAndHyperplane}
		Let $m\in\NNnot$ and suppose $q=(q')^{\rho+1}$.
		Consider an $(m-1)$-dimensional $q'$-subgeometry $\HyperSub_1$ of $\pg(m,q)$ and an $(m-2)$ $q'$-subgeometry $\HyperSub_2\subseteq\HyperSub_1$; define $\Sigma_i:=\vspanq{\HyperSub_i}$.
		Let $\Sub_1$ and $\Sub_2$ be two distinct $m$-dimensional $q'$-subgeometries, both containing $\HyperSub_1$ and a point $F\notin\Sigma_1$, and suppose that $\Pi$ is an $(m-1)$-space through $\Sigma_2$ not equal to $\Sigma_1$ and not containing $F$.
		Then $\Pi$ cannot intersect both $\Sub_1$ and $\Sub_2$ in an $(m-1)$-dimensional $q'$-subgeometry.
	\end{lm}
	\begin{proof}
		Suppose that the contrary is true.
		Choose a point $F'\in\HyperSub_1\setminus\HyperSub_2$.
		Then the line $FF'$ intersects $\Pi$ in a point $P$.
		As $\Pi$ intersects both $\Sub_1$ and $\Sub_2$ in a $q'$-subgeometry of maximal dimension, $P$ has to be a point of both $\Sub_1$ as $\Sub_2$.
		Moreover, as both these subgeometries contain $\HyperSub_1\ni F'$ and $F$, the unique $q'$-subline containing $F$, $F'$ and $P$ has to be contained in both $\Sub_1$ and $\Sub_2$.
		By Lemma \ref{Lm_UniqueSubThroughHyperSubAndLine}, this would imply that $\Sub_1=\Sub_2$, a contradiction.
	\end{proof}
	
	Finally, the following lemma is a direct consequence of the \emph{strong blocking set approach} (see Subsection \ref{Subsec_KnownApproaches}).
	
	\begin{lm}[{\cite[Section $3$]{DavydovEtAl1}}]\label{Lm_SubIsRhoSaturating}
		Let $m\in\NN$ and let $q=(q')^{\rho+1}$.
		Then any $m$-dimensional $q'$-subgeometry of $\pg(m,q)$ is a $\rho$-saturating set of $\pg(m,q)$.
	\end{lm}
	\begin{proof}
		Let $\Sub$ be an $m$-dimensional $q'$-subgeometry of $\pg(m,q)$.
		If $m\leq\rho$, we can simply choose a base of the subgeometry $\Sub$; naturally, such a set of points spans the whole space $\pg(m,q)$.
		As this base contains $m+1\leq\rho+1$ points, the proof is done.
		
		If $m\geq\rho+1$, then the proof is exactly the same as described in \cite[proof of Theorem $3.2$]{DavydovEtAl1}.
	\end{proof}
	
	\subsection{Constructing the saturating set}
	
    We will construct a small $\rho$-saturating set of $\pg(\n,q)$ by making use of the following observation.
    
    \begin{lm}\label{Lm_FlowerSaturates}
        Let $m\in\NNnot$, $0<\rho<m$ and $q=(q')^{\rho+1}$.
        Suppose that $\Flower:=\{\tau_1,\dots,\tau_{\rho+1}\}$ is an $(m-\rho)$-flower of $\pg(m,q)$ with certain pistil $\Sigma$.
        Let $\HyperSub\subseteq\Sigma$ be an $(m-\rho-1)$-dimensional $q'$-subgeometry and consider, for every $j\in\{1,2,\dots,\rho+1\}$, $j$ distinct, $\restr{\widetilde{\varphi}}{\tau_j}$-independent elements $\Sub_j^{(1)}\setminus\HyperSub,\Sub_j^{(2)}\setminus\HyperSub,\dots,\Sub_j^{(j)}\setminus\HyperSub$ of $\llinesSub^{\tau_j}$, all sharing a point $F_j\in\ppointsSub^{\tau_j}$.
        Then the point set
        \[
            \SubSat:=\bigcup_{j=1}^{\rho+1}\bigcup_{k=1}^j\left(\Sub_j^{(k)}\setminus\HyperSub\right)
        \]
        $\rho$-saturates all points of $\pg(m,q)$ that do not lie in the span of any $\rho$ petals of $\Flower$.
    \end{lm}
    \begin{proof}
        Let $P$ be a point not contained in the span of any $\rho$ petals of $\Flower$.
        Define, for every $j\in\{1,2,\dots,\rho+1\}$, $\Pi_j:=\vspan{\tau_j,\tau_{j+1},\dots,\tau_{\rho+1}}$.
        Note that, by definition of a flower (Definition \ref{Def_Flower}), $\Pi_1$ is equal to the whole space $\pg(m,q)$, hence $P\in\Pi_1$.
        Furthermore, consider the $(m-\rho)$-space $\pi_0:=\vspan{\Sigma,P}$ and the point set $\mathcal{T}_0:=\{P\}$.
        One can now iterate through the following process, for $j$ going from $1$ to $\rho$.
        \begin{enumerate}
            \item Observe that $\pi_{j-1}$ and $\tau_j$ are distinct $(m-\rho)$-subspaces through $\Sigma$, contained in $\Pi_j$ but not contained in $\Pi_{j+1}$.
            As the latter is a hyperplane of $\Pi_j$, $\vspan{\pi_{j-1},\tau_j}$ intersects $\Pi_{j+1}$ in an $(m-\rho)$-space $\pi_j:=\vspan{\pi_{j-1},\tau_j}\cap\Pi_{j+1}$.
            \item Observe that $\pi_{j-1}$, $\tau_j$ and $\pi_j$ are three distinct $(m-\rho)$-spaces through $\Sigma$, spanning an $(m-\rho+1)$-space.
            Furthermore, $\mathcal{T}_{j-1}\subseteq\ppointsSub^{\pi_{j-1}}$ is a $(j-1)$-dimensional affine $\restr{\widetilde{\varphi}}{\pi_{j-1}}$-subspace.
            By Lemma \ref{Lm_ProjectionAndShadowGetBigger}, there exists a $\Sub_j^{(k)}\setminus\HyperSub$ in $\llinesSub^{\tau_j}$ such that $\mathcal{T}_j:=\prArg{\tau_j}{\pi_j}{\mathcal{T}_{j-1}}{\Sub_j^{(k)}\setminus\HyperSub}$ is a $j$-dimensional affine $\restr{\widetilde{\varphi}}{\pi_j}$-subspace.
            \item Note that any point $T_j\in\mathcal{T}_j$ lies in the span of a point of $\SubSat\cap\tau_j$ and a point of $\mathcal{T}_{j-1}$.
            Indeed, by definition of $\pr{\tau_j}{\pi_j}{\mathcal{T}_{j-1}}$, there has to exist a point $T_j'\in\mathcal{T}_{j-1}$ such that $T_j\in\prArg{\tau_j}{\pi_j}{T_j'}{\Sub_j^{(k)}\setminus\HyperSub}$.
            Hence, there exists a point of $\Sub_j^{(k)}\setminus\HyperSub$ that is projected from $T_j'$ onto $T_j$.
        \end{enumerate}
        Eventually, $\mathcal{T}_\rho\subseteq\ppointsSub^{\pi_\rho}$ is a $\rho$-dimensional affine $\restr{\widetilde{\varphi}}{\pi_\rho}$-subspace.
        In other words, the images of all points in $\mathcal{T}_\rho$ under $\restr{\widetilde{\varphi}}{\pi_\rho}$ form a hyperplane of $\ag\big(\rho+1,(q')^{m-\rho}\big)$.
        Furthermore, note that $\mathcal{T}_\rho\subseteq\pi_\rho\subseteq\Pi_{\rho+1}=\tau_{\rho+1}$.
        As $\SubSat\cap\tau_{\rho+1}$ is a union of $\rho+1$ concurrent, $\restr{\widetilde{\varphi}}{\tau_{\rho+1}}$-independent lines, there has to exist a point $Q_{\rho+1}\in\mathcal{T}_\rho\cap\SubSat\cap\tau_{\rho+1}$, since any union of $\rho+1$ concurrent, independent lines of $\ag\big(\rho+1,(q')^{m-\rho}\big)$ meets any hyperplane in at least one point.
        
        By recursively backtracking the observation obtained in step $3.$, we conclude that $Q_{\rho+1}$ lies in $\vspan{Q_\rho,Q_{\rho-1},\dots,Q_1,P}$, with $Q_j\in\SubSat\cap\tau_j$ ($j\in\{1,2,\dots,\rho+1\}$).
        This implies that $P\in\vspan{Q_1,Q_2,\dots,Q_{\rho+1}}$, as no point of $\{Q_1,Q_2,\dots,Q_{\rho+1}\}$ can lie in the span of the others (else $\Flower$ would not be an $(m-\rho)$-flower).
    \end{proof}
    
    By the lemma above, we can find a relatively small point set that $\rho$-saturates `most' of the points of $\pg(\n,q)$.
    We could end our quest right here and now, by copying smaller versions of similar point sets in the span of any $\rho$ petals of $\Flower$.
    However, as this would dramatically increase the size of the saturating set, we need to optimise the construction.
    Hence, to compensate for the restricted $\rho$-saturating capabilities described by the lemma above, we construct a $\rho$-saturating set as a mix of several flowers.
    To do this, we fix following notation and definition.
	
	\begin{nt}
	    Define the value $\lambda:=\min\{\rho,\n-\rho\}$.
	\end{nt}
	
	\begin{df}\label{Def_LowerHigherMaps}
	    Let $0<\rho<\n$.
	    For every $i\in\{1,\dots,\lambda\}$, define a map
	    \begin{align*}
	        \fn{\cdot}{i}&:\{\rho+2-\lambda,\dots,\rho+1\}\rightarrow\{\rho+2-\lambda,\dots,\rho+1\}\\
	        &:j\mapsto\fn{j}{i}:=\begin{cases}j+i-1&\quad\textnormal{if }j+i-1\leq\rho+1\textnormal{,}\\\rho+2-i&\quad\textnormal{otherwise.}\end{cases}
	    \end{align*}
    \end{df}
    
    As the map above could induce some confusion, we will give the reader an intuition of Construction \ref{Constr_Subgeometric} (see below) before plunging into the technical details.
    
    As said before, the main construction will be built by making use of a mix of multiple flowers.
    These flowers will be stacked upon each other, forming a total of $\lambda$ `layers', in the sense that
    \begin{itemize}
        \item the `largest' layer (layer $i=1$) is an $(\n-\rho)$-flower with $\rho+1$ petals; the petals will be numbered $1,2,\dots,\rho+1$,
        \item within that layer, we consider an $(\n-\rho-1)$-flower with $\rho+1$ petals (layer $i=2$), such that each of the petals is contained in a unique petal of the layer above,
        \item within that layer, we consider an $(\n-\rho-2)$-flower with $\rho+1$ petals (layer $i=3$), such that each of the petals is contained in a unique petal of the layer above,
        \item ...
        \item the `smallest' layer (layer $i=\lambda$) is an $(\n-\rho-\lambda+1)$-flower with $\rho+1$ petals, such that each of the petals is contained in a unique petal of the layer above.
    \end{itemize}
    In this way, we obtain a large flower consisting of `layered' petals numbered $1,2,\dots,\rho+1$.
    
    Inspired by Lemma \ref{Lm_FlowerSaturates}, we now choose a set of concurrent, $\widetilde{\varphi}$-independent $q'$-subgeometries in certain layers of each of the petals.
    The number of such subgeometries will depend on the number of the layer ($i$) and the number of the petal ($j$).
    If $j\leq\rho+1-\lambda$, we will choose $j$ concurrent, $\widetilde{\varphi}$-independent $q'$-subgeometries in the top layer ($i=1$) of petal $j$, and none in any of its other layers.
    If $j>\rho+1-\lambda$, then the value $\fn{j}{i}$ equals the number of concurrent, $\widetilde{\varphi}$-independent $q'$-subgeometries we will choose in layer $i$ of petal $j$.
    To elaborate, if $j>\rho+1-\lambda$, we will choose
    \begin{itemize}
        \item precisely $\fn{j}{1}=j$ concurrent, $\widetilde{\varphi}$-independent $q'$-subgeometries in the top layer of petal $j$,
        \item precisely $\fn{j}{2}=j+1$ concurrent, $\widetilde{\varphi}$-independent $q'$-subgeometries in the next layer ($i=2$) of petal $j$,
        \item ...
        \item precisely $\fn{j}{\rho+2-j}=\rho+1$ concurrent, $\widetilde{\varphi}$-independent $q'$-subgeometries in the next layer ($i=\rho+2-j$) of petal $j$,
        \item precisely $\fn{j}{\rho+3-j}=j-1$ concurrent, $\widetilde{\varphi}$-independent $q'$-subgeometries in the next layer ($i=\rho+3-j$) of petal $j$,
        \item ...
        \item precisely $\fn{j}{\lambda}=\rho+2-\lambda$ concurrent, $\widetilde{\varphi}$-independent $q'$-subgeometries in the bottom layer ($i=\lambda$) of petal $j$.
    \end{itemize}
    
    We now formalise the intuitive construction above and hence introduce the main construction of this article.
    Be sure to keep Figure \ref{Fig_ExampleOfConstruction} at hand for a visualisation of an example case with three two-layered petals.
	
	\begin{constr}\label{Constr_Subgeometric}
	    Let $0<\rho<\n$ and let $q=(q')^{\rho+1}$.
	    Suppose $\{\HyperSub_1,\dots,\HyperSub_\lambda\}$ is a set of $q'$-subgeometries and suppose $\{\Sigma_1,\dots,\Sigma_\lambda\}$ is a set of subspaces of $\pg(\n,q)$ with the following properties.
	    \begin{itemize}
	        \item For every $i\in\{1,\dots,\lambda\}$, $\HyperSub_i$ is an $(\n-\rho-i)$-dimensional $q'$-subgeometry such that $\Sigma_i=\vspanq{\HyperSub_i}$.
	        \item $\HyperSub_1\supseteq\HyperSub_2\supseteq\dots\supseteq\HyperSub_\lambda$, hence $\Sigma_1\supseteq\Sigma_2\supseteq\dots\supseteq\Sigma_\lambda$.
	    \end{itemize}
	    Moreover, consider a set of flowers $\{\Flower_1,\dots,\Flower_\lambda\}$ with the following properties.
	    \begin{itemize}
	        \item For every $i\in\{1,\dots,\lambda\}$, $\Flower_i:=\left\{\tau_{i1},\dots,\tau_{i(\rho+1)}\right\}$ is an $(\n-\rho-i+1)$-flower with pistil $\Sigma_i$.
	        \item For every $j\in\{1,\dots,\rho+1\}$, $\tau_{1j}\supseteq\tau_{2j}\supseteq\dots\supseteq\tau_{\lambda j}$.
	    \end{itemize}
	    Now define, for every $j\in\{1,\dots,\rho+1\}$,
	    \[
	        \ppi_j:=\begin{cases}
	            \bigcup_{k=1}^j\left(\Sub_{1j}^{(k)}\setminus\HyperSub_1\right)&\textnormal{if }j\leq\rho+1-\lambda\textnormal{,}\\
	            \bigcup_{i=1}^\lambda\bigcup_{k=1}^{\fn{j}{i}}\left(\Sub_{ij}^{(k)}\setminus\HyperSub_i\right)&\textnormal{if }j>\rho+1-\lambda\textnormal{,}
	        \end{cases}
	    \]
	    where $\Sub_{ij}^{(1)}\setminus\HyperSub_i,\Sub_{ij}^{(2)}\setminus\HyperSub_i,\dots,\Sub_{ij}^{\left(\fn{j}{i}\right)}\setminus\HyperSub_i$ are $\fn{j}{i}$ distinct, $\restr{\widetilde{\varphi}}{\tau_{ij}}$-independent elements of $\llinesSubIndex{i}^{\tau_{ij}}$, all sharing a point $F_{ij}\in\ppointsSubIndex{i}^{\tau_{ij}}\setminus\tau_{(i+1)j}$ ($i\in\{1,\dots,\lambda\}$, $\tau_{(\lambda+1)j}:=\emptyset$).
	    
	    Finally, define
	    \[
	        \ppi_1':=\begin{cases}
	            \bigcup_{i=2}^\lambda\left(\Sub_{i1}'\setminus\HyperSub_i\right)&\textnormal{if }q'=2\textnormal{,}\\
	            \emptyset&\textnormal{if }q'\neq2\textnormal{,}
	        \end{cases}
	    \]
	    with $\Sub_{i1}'\setminus\HyperSub_i$ an element of $\llinesSubIndex{i}^{\tau_{i1}}$ with the property that $\vspanq{\Sub_{i1}'}$ intersects $\Sub_{(i-1)1}'$ only in $\HyperSub_i$ ($i\in\{2,\dots,\lambda\}$, $\Sub_{11}':=\Sub_{11}^{(1)}$).
	\end{constr}
	
	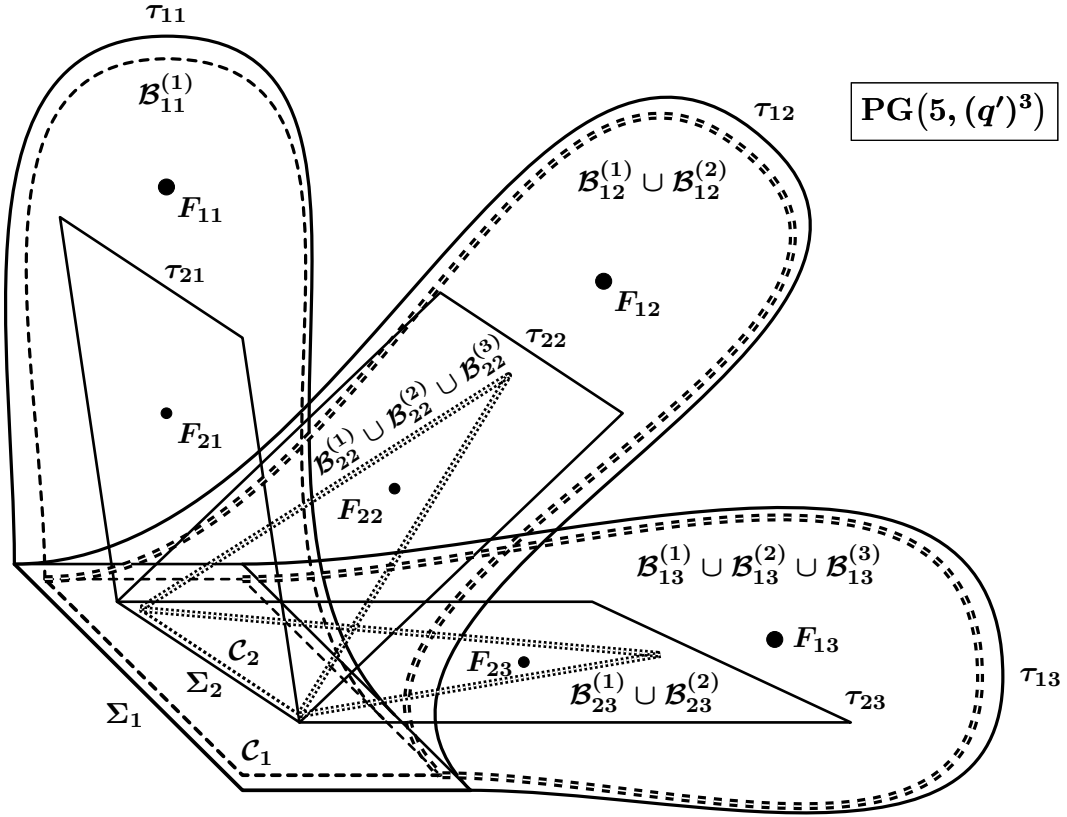
\begin{figure}
        \begin{center}\begin{tikzpicture}
            
            \node[draw,fill=none,anchor=west] at (4,4) {\large\textbf{PG}$\boldsymbol{\big(5,(q')^3\big)}$};
			\node[draw,fill=none,anchor=west] at (4,4) {\large\textbf{PG}$\boldsymbol{\big(5,(q')^3\big)}$}; 
			
			\draw[line width=1pt, line join=round, line cap=round] (-7,-2) -- (-4,-2) -- (-1,-5);
			\draw[line width=1.4pt, line join=round, line cap=round] (-1,-5) -- (-4,-5) -- (-7,-2);
			\node[draw=none, fill=none] at (-5.55,-4) {$\boldsymbol{\Sigma_1}$};
			\node[draw=none, fill=none] at (-4.5,-3.6) {$\boldsymbol{\Sigma_2}$};
			
			\draw[line width=1pt, dashed, line join=round, line cap=round] (-6.6,-2.2) -- (-4,-2.2) -- (-1.4,-4.8);
			\draw[line width=1.4pt, dashed, line join=round, line cap=round] (-1.4,-4.8) -- (-4,-4.8) -- (-6.6,-2.2);
			\node[draw=none, fill=none] at (-3.8,-4.5) {$\boldsymbol{\HyperSub_1}$};
			
			\draw[line width=1.2pt, densely dotted, line join=round] (-5.35,-2.6) -- (-3.25,-4);
			\node[draw=none, fill=none] at (-3.97,-3.17) {$\boldsymbol{\HyperSub_2}$};
			
			\draw[line width=1.2pt, line join=round, line cap=round] (-7,-2) to [out=90, in=180] (-5,5) to [out=0, in=135] (-2,-4);
			\node[draw=none, fill=none] at (-5,5.3) {$\boldsymbol{\tau_{11}}$};
			
			\draw[fill=black] (-5,3) circle (3pt);
			\node[draw=none, fill=none, anchor=north west] at (-5,3) {$\boldsymbol{F_{11}}$};
			
			\draw[line width=1.2pt, dashed, line join=round, line cap=round] (-6.6,-2.2) to [out=90, in=180] (-5,4.7) to [out=0, in=135] (-2.15,-4.05);
			\node[draw=none, fill=none] at (-5,4.25) {$\boldsymbol{\Sub_{11}^{(1)}}$};
			
			\draw[line width=1pt, line join=round, line cap=round] (-5.65,-2.5) -- (-3.25,-4.1) -- (-4,1) -- (-6.4,2.6) -- cycle;
			\node[draw=none, fill=none] at (-4.75,1.85) {$\boldsymbol{\tau_{21}}$};
			
			\draw[fill=black] (-5,0) circle (2pt);
			\node[draw=none, fill=none, anchor=north west] at (-5,0) {$\boldsymbol{F_{21}}$};
			
			\draw[line width=1.2pt, line join=round, line cap=round] (-7,-2) to [out=0, in=135] (3, 3.5) to [out=315, in=135] (-1,-5);
			\node[draw=none, fill=none] at (3,4) {$\boldsymbol{\tau_{12}}$};
			
			\draw[fill=black] (0.75,1.75) circle (3pt);
			\node[draw=none, fill=none, anchor=north west] at (0.75,1.75) {$\boldsymbol{F_{12}}$};
			
			\draw[line width=1.2pt, double, dashed, line join=round] (-6.6,-2.2) to [out=0, in=135] (2.8, 3.3) to [out=315, in=135] (-1.4,-4.8);
			\node[draw=none, fill=none] at (1.4,3.1) {$\boldsymbol{\Sub_{12}^{(1)}\cup\Sub_{12}^{(2)}}$};
			
			\draw[line width=1pt, line join=round, line cap=round] (-3.25,-4.1) -- (1,0) -- (-1.4,1.6) -- (-5.65,-2.5);
			\node[draw=none, fill=none] at (0,1) {$\boldsymbol{\tau_{22}}$};
			
			\draw[line width=1pt, double, densely dotted, line join=round] (-3.25,-4) -- (-0.5,0.5) -- (-5.35,-2.6);
			\node[draw=none, fill=none, rotate around={33:(0,0)}] at (-1.8,0.1) {\small$\boldsymbol{\Sub_{22}^{(1)}\cup\Sub_{22}^{(2)}\cup\Sub_{22}^{(3)}}$};
			
			\draw[fill=black] (-2,-1) circle (2pt);
			\node[draw=none, fill=none, anchor=north east] at (-2,-1) {$\boldsymbol{F_{22}}$};
			
			\draw[line width=1.2pt, line join=round, line cap=round] (-4,-2) to [out=0, in=90] (6,-3.5) to [out=270, in=0] (-1,-5);
			\node[draw=none, fill=none] at (6.5,-3.5) {$\boldsymbol{\tau_{13}}$};
			
			\draw[fill=black] (3,-3) circle (3pt);
			\node[draw=none, fill=none, anchor=west] at (3.1,-3) {$\boldsymbol{F_{13}}$};
			
			\draw[line width=1.2pt, double, dashed, line join=round] (-4,-2.2) to [out=0, in=90] (5.7,-3.5) to [out=270, in=0] (-1.4,-4.8);
			\node[draw=none, fill=none] at (2.8,-2) {$\boldsymbol{\Sub_{13}^{(1)}\cup\Sub_{13}^{(2)}\cup\Sub_{13}^{(3)}}$};
			
			\draw[line width=1pt, line join=round, line cap=round] (-3.25,-4.1) -- (4,-4.1) -- (0.6,-2.5) -- (-5.65,-2.5);
			\node[draw=none, fill=none] at (4.2,-3.8) {$\boldsymbol{\tau_{23}}$};
			
			\draw[line width=1pt, double, densely dotted, line join=round] (-3.25,-4) -- (1.5,-3.2) -- (-5.35,-2.6);
			\node[draw=none, fill=none] at (1.3,-3.7) {$\boldsymbol{\Sub_{23}^{(1)}\cup\Sub_{23}^{(2)}}$};
			
			\draw[fill=black] (-0.3,-3.3) circle (2pt);
			\node[draw=none, fill=none, anchor=east] at (-0.3,-3.32) {$\boldsymbol{F_{23}}$};
			
        \end{tikzpicture}\end{center}
        \caption{A visualisation of Construction \ref{Constr_Subgeometric} in case $\n=5$ and $\rho=2$; we observe two stacked flowers, resulting in three two-layered petals.
        The petal $\tau_{11}$ has a number $j=1$ not exceeding $\rho+1-\lambda=1$.
        The petals with number $j=2$ correspond to $\fn{2}{1}=2$ chosen $q'$-subgeometries in the top layer and $\fn{2}{2}=3$ chosen $q'$-subgeometries in the bottom layer (increasing).
        The petals with number $j=3$ correspond to $\fn{3}{1}=3$ chosen $q'$-subgeometries in the top layer and $\fn{3}{2}=2$ chosen $q'$-subgeometries in the bottom layer (decreasing).}
        \label{Fig_ExampleOfConstruction}
	\end{figure}
	
	\begin{lm}\label{Lm_SizePii1}
	    Consider Construction \ref{Constr_Subgeometric}.
	    Then
	    \[
	        |\ppi_1'|=\begin{cases}(2^{\lambda-1}-1)\cdot2^{\n-\rho-\lambda+1}&\textnormal{if }q'=2\textnormal{,}\\
	        0&\textnormal{if }q'\neq2\textnormal{.}\end{cases}
	    \]
	\end{lm}
	
	\begin{lm}\label{Lm_SizePiij}
	    Consider Construction \ref{Constr_Subgeometric}.
	    Let $j\in\{1,\dots,\rho+1\}$.
	    If $j\leq\rho+1-\lambda$, then
	    \[
	        |\ppi_j|=j(q')^{\n-\rho}-(j-1)\textnormal{.}
	    \]
	    If $j>\rho+1-\lambda$, then one can find a set $\ppi_j$ such that
	    \[
	        |\ppi_j|=j(q')^{\n-\rho}+\sum_{k=1}^{\rho+1-j}(j-1+k)(q')^{\n-\rho-k}+\sum_{k=\rho+2-j}^{\lambda-1}(\rho-k)(q')^{\n-\rho-k}-\frac{\lambda(2\rho-\lambda+1)}{2}\textnormal{.}
	    \]
	\end{lm}
	\begin{proof}
	    If $j\leq\rho+1-\lambda$, this result is easily obtained, as distinct elements of $\llinesSubIndex{1}^{\tau_{1j}}$ can share at most one point.
	    Hence, assume $j>\rho+1-\lambda$.
	    To minimalise the size of $\ppi_j$, we can always choose $\Sub_{ij}^{(1)}$ to be a subspace of $\Sub_{(i-1)j}^{(1)}$, for every $i\in\{2,\dots,\lambda\}$.
	    In this way, keeping the nature of $\fn{\cdot}{\cdot}$ in mind (see Definition \ref{Def_LowerHigherMaps}), we obtain the following:
	    \begin{align*}
	        |\ppi_j|&=j(q')^{\n-\rho}&&-(j-1)\\
	        &+j(q')^{\n-\rho-1}&&-j\\
	        &+(j+1)(q')^{\n-\rho-2}&&-(j+1)\\
	        &\;\;\vdots&&\;\;\vdots\\
	        &+\rho(q')^{\n-\rho-(\rho+1-j)}&&-\rho\\
	        &+(j-2)(q')^{\n-\rho-(\rho+2-j)}&&-(j-2)\\
	        &+(j-3)(q')^{\n-\rho-(\rho+3-j)}&&-(j-3)\\
	        &\;\;\vdots&&\;\;\vdots\\
	        &+(\rho+1-\lambda)(q')^{\n-\rho-(\lambda-1)}&&-(\rho+1-\lambda)\textnormal{.}
	    \end{align*}
	    Viewing the expression above as a polynomial in $q'$, the corresponding constant term equals
	    \begin{align*}
	        -\big((\rho+1-\lambda)+\dots+\rho\big)&=\frac{(\rho-\lambda)(\rho+1-\lambda)}{2}-\frac{\rho(\rho+1)}{2}\\
	        &=-\frac{\lambda(2\rho-\lambda+1)}{2}\textnormal{.}\qedhere
	    \end{align*}
	\end{proof}
	
	\begin{lm}\label{Lm_SizeSumPiij}
	    Consider Construction \ref{Constr_Subgeometric}.
	    Then we can find sets $\ppi_1,\dots,\ppi_{\rho+1}$ such that
	    \[
	        \sum_{i=1}^{\rho+1}|\ppi_i|=\frac{(\rho+1)(\rho+2)}{2}(q')^{\n-\rho}+\sum_{j=1}^{\lambda-1}a(\n,\rho,j)(q')^{\n-\rho-j}-c(\n,\rho)\textnormal{,}
	    \]
	    with
	    \[
	        a(\n,\rho,j):=\frac{\lambda(2\rho-\lambda+2j+1)-j(3j+1)}{2}\;\;\;\textnormal{and}\;\;\;c(\n,\rho):=\frac{\rho(\rho+1)+\lambda(\lambda-1)(2\rho-\lambda+1)}{2}\textnormal{.}
	    \]
	\end{lm}
	\begin{proof}
	    Let $\ppi_1,\dots,\ppi_{\rho+1}$ be sets of size equal to the values described in Lemma \ref{Lm_SizePiij}.
	    Interpret $\sum_{i=1}^{\rho+1}|\ppi_i|$ as a polynomial in $q'$ of degree $\n-\rho$; let $a(\n,\rho,j)$ be the coefficient corresponding to $(q')^{\n-\rho-j}$ ($j\in\{0,1,\dots,\n-\rho-1\}$) and let $-c(\n,\rho)$ be the constant term.
	    
	    It is clear that $a(\n,\rho,0)=\sum_{i=1}^{\rho+1}i=\frac{(\rho+1)(\rho+2)}{2}$.
	    Furthermore, we can deduce that
	    \begin{align*}
	        a(\n,\rho,1)&=\underbrace{(\rho+2-\lambda)+(\rho+3-\lambda)+\dots+(\rho)}_{\textnormal{arising from }\ppi_{\rho+2-\lambda},\ppi_{\rho+3-\lambda},\,\dots\,,\ppi_\rho}+\underbrace{(\rho-1)}_{\textnormal{arising from }\ppi_{\rho+1}}\\
	        &=\frac{\lambda\big(2\rho-\lambda+3\big)-4}{2}\textnormal{,}\\
	        a(\n,\rho,2)&=\underbrace{(\rho+3-\lambda)+(\rho+4-\lambda)+\dots+(\rho)}_{\textnormal{arising from }\ppi_{\rho+2-\lambda},\ppi_{\rho+3-\lambda},\,\dots\,,\ppi_{\rho-1}}+\underbrace{2(\rho-2)}_{\textnormal{arising from }\ppi_\rho\textnormal{ and }\ppi_{\rho+1}}\\
	        &=\frac{\lambda\big(2\rho-\lambda+5\big)-14}{2}\textnormal{,}\\
	        &\;\;\vdots\\
	        a(\n,\rho,j)&=\underbrace{(\rho+j+1-\lambda)+(\rho+j+2-\lambda)+\dots+(\rho)}_{\textnormal{arising from }\ppi_{\rho+2-\lambda},\ppi_{\rho+3-\lambda},\,\dots\,,\ppi_{\rho+1-j}}+\underbrace{j(\rho-j)}_{\textnormal{arising from }\ppi_{\rho+2-j},\,\dots\,,\ppi_{\rho+1}}\\
	        &=\frac{\lambda(2\rho-\lambda+2j+1)-j(3j+1)}{2}\textnormal{,}\\
	        &\;\;\vdots\\
	        a(\n,\rho,\lambda-1)&=\underbrace{(\rho)}_{\textnormal{arising from }\ppi_{\rho+2-\lambda}}+\underbrace{(\lambda-1)(\rho+1-\lambda)}_{\textnormal{arising from }\ppi_{\rho+3-\lambda},\,\dots\,,\ppi_{\rho+1}}\\
	        &=\rho\lambda-\lambda^2+2\lambda-1\textnormal{,}\\
	        a(\n,\rho,\lambda)&=\dots=a(\n,\rho,\n-\rho-1)=0\textnormal{,}\\
	        -c(\n,\rho)&=\underbrace{-1-2-3-\dots-(\rho-\lambda)}_{\textnormal{arising from }\ppi_2,\ppi_3,\,\dots\,,\ppi_{\rho+1-\lambda}}+\underbrace{\lambda\left(-\frac{\lambda(2\rho-\lambda+1)}{2}\right)}_{\textnormal{arising from }\ppi_{\rho+2-\lambda},\ppi_{\rho+3-\lambda},\,\dots\,,\ppi_{\rho+1}}\\
	        &=-\frac{\rho(\rho+1)+\lambda(\lambda-1)(2\rho-\lambda+1)}{2}\textnormal{.}\qedhere
	    \end{align*}
	\end{proof}
	
	\begin{lm}\label{Lm_SubgeometricIsSaturating}
	    Consider Construction \ref{Constr_Subgeometric}.
	    Then the point set
	    \[
	        \SubSat_{(\n,\rho)}:=\ppi_1'\cup\bigcup_{j=1}^{\rho+1}\ppi_j
	    \]
	    $\rho$-saturates all points of $\pg(\n,q)\setminus\Sigma_1$.
	\end{lm}
	\begin{proof}
	    Let $P$ be an arbitrary point of $\pg(\n,q)\setminus\Sigma_1$ and let
	    \[
	        \mu:=\min\big\{|\Flower|:\Flower\subseteq\Flower_1,P\in\vspanbig{\tau:\tau\in\Flower}\big\}\in\{1,2,\dots,\rho+1\}\textnormal{.}
	    \]
	    Hence, there exists an $(\n-\rho)$-flower $\Flower_1':=\{\tau_{11}',\tau_{12}',\dots,\tau_{1\mu}'\}\subseteq\Flower_1$ with pistil $\Sigma_1$ such that $P$ lies in the span of all $\mu$ petals of $\Flower_1'$, but does not lie in the span of any $\mu-1$ petals of $\Flower_1'$.
	    
	    For each petal $\tau_{1j}'$, $j\in\{2,3,\dots,\mu-1\}$, only the points of $\SubSat_{(\n,\rho)}$ in the top layer ($i=1$) of $\tau_{1j}'$ will be used to prove point saturation.
	    As a consequence, we can assume without loss of generality that $\tau_{1j}'=\tau_{1j}$ for every $j\in\{2,3,\dots,\mu-1\}$.
	    If $q'>2$, the same can be said about petal $\tau_{11}'$.
	    If $q'=2$, however, two possibilities can occur:
	    \begin{enumerate}[label=\roman*]
	        \item either there exist at least two $(\n-\rho)$-dimensional $q'$-subgeometries $\Sub_1$ and $\Sub_2$ in $\tau_{11}'$, both containing $\HyperSub_1$ and a point $F\in\tau_{11}'\setminus\Sigma_1$, such that $\left(\Sub_1\cup\Sub_2\right)\setminus\HyperSub_1\subseteq\SubSat_{(\n,\rho)}$, or
	        \item $\tau_{11}'=\tau_{11}$.
	    \end{enumerate}
	    Note that for both possibilities i and ii, there exists one $(\n-\rho)$-dimensional $q'$-subgeometry $\Sub$ in $\tau_{11}'$ containing $\HyperSub_1$ such that $\Sub\setminus\HyperSub_1\subseteq\SubSat_{(\n,\rho)}$; this is the only property needed of petal $\tau_{11}'$ in Case $1$, Case $2$ and Case $3$ (step $1.$ and $2.$) below.
	    Only in Case $3$ (step $3.$) a distinction between possibility i and ii has to be made.
	    In light of this, we will, for now, assume that $\tau_{11}'=\tau_{11}$, and will remove this assumption in the third step of Case $3$.
	    Finally, we may assume that $\tau_{1\mu}'\in\{\tau_{1\mu},\tau_{1(\mu+1)},\dots,\tau_{1(\rho+1)}\}$, hence there has to exist a $j'\geq\mu$ such that $\tau_{1\mu}'=\tau_{1j'}$.
		In conclusion, assume that $\Flower_1'=\{\tau_{11},\tau_{12},\dots,\tau_{1(\mu-1)},\tau_{1j'}\}$ for a certain $j'\geq\mu$.
		
	    If $\mu=\rho+1$, the proof follows immediately due to Lemma \ref{Lm_FlowerSaturates}.
	    We consider three cases, depending on the other possible values of $\mu$.
	    
	    \bigskip
		\underline{Case $1$}: $\mu=1$.
		
		\bigskip
		In this case, $P$ is contained in $\tau_{11}$, which is an $(\n-\rho)$-dimensional subspace containing $\Sub_{11}^{(1)}\setminus\HyperSub_1\subseteq\SubSat_{(\n,\rho)}$.
		By Lemma \ref{Lm_SubIsRhoSaturating}, there exists an $r$-subspace $\pi$ of $\Sub_{11}^{(1)}$, $0<r\leq\rho$, such that $\vspanq{\pi}$ contains $P$. As $P\notin\Sigma_1$, $\pi\setminus\Sigma_1$ has to be isomorphic to $\ag(r,q')$, hence we can easily find $r+1$ points in $\pi\setminus\Sigma_1\subseteq\Sub_{11}^{(1)}\setminus\Sigma_1$ spanning $\vspanq{\pi}$.
		
		\bigskip
		\underline{Case $2$}: $1<\mu\leq\rho+1-\lambda$.
		
		\bigskip
	    Note that the occurrence of this case implies that $\lambda=\n-\rho$.
		
		We can choose $\n-\rho+1$ points of $\Sub_{1j'}^{(1)}\setminus\HyperSub_1$ spanning the subspace $\tau_{1j'}\supseteq\Sigma_1$, and one point in each set $\Sub_{11}^{(1)}\setminus\HyperSub_1,\dots,\Sub_{1(\mu-1)}^{(1)}\setminus\HyperSub_1$.
		These choices result in a total of $(\n-\rho+1)+(\mu-1)\leq(\n-\rho+1)+(\rho-\lambda)=\rho+1$ points spanning $\vspan{\tau_{11},\tau_{12},\dots,\tau_{1(\mu-1)},\tau_{1j'}}\ni P$.
		
		\bigskip
		\underline{Case $3$}: $\rho+1-\lambda<\mu\leq\rho$.
		
		\bigskip
		Consider the following series of steps.
		\begin{enumerate}
		    \item For every $j\in\{1,2,\dots,\mu-1\}\cup\{j'\}$, define $\widetilde{\tau}_{2j}:=\vspan{\Sigma_2,F_{1j}}$ and consider, for every $k\in\{1,2,\dots,j\}$, the $(\n-\rho-1)$-dimensional $q'$-subgeometry $\BigSub_{2j}^{(k)}:=\Sub_{1j}^{(k)}\cap\vspan{\Sigma_2,F_{1j}}$.
		    In this way, we obtain, for every $j\in\{1,2,\dots,\mu-1\}\cup\{j'\}$, a union $\BigSub_{2j}^{(1)}\cup\BigSub_{2j}^{(2)}\cup\dots\cup\BigSub_{2j}^{(j)}$ of $j$ distinct, $\restr{\widetilde{\varphi}}{\widetilde{\tau}_{2j}}$-independent $(\n-\rho-1)$-dimensional $q'$-subgeometries in an $(\n-\rho-1)$-dimensional subspace $\widetilde{\tau}_{2j}$ of $\tau_{1j}$, each containing $\HyperSub_2$ and sharing the point $F_{1j}$.
		    \item Consider the union $\Sub_{2j'}^{(1)}\cup\Sub_{2j'}^{(2)}\cup\dots\cup\Sub_{2j'}^{\left(\fn{j'}{2}\right)}$ of $\fn{j'}{2}$ distinct, $\restr{\widetilde{\varphi}}{\tau_{2j'}}$-independent $(\n-\rho-1)$-dimensional $q'$-subgeometries in the $(\n-\rho-1)$-dimensional subspace $\tau_{2j'}\neq\widetilde{\tau}_{2j'}$ of $\tau_{1j'}$, each containing $\HyperSub_2$ and sharing the point $F_{2j'}$.
		    It is clear that $\tau_{2j'}$ and $\widetilde{\tau}_{2j'}$ span the space $\tau_{1j'}$, as these are distinct hyperplanes of the latter space.
    		\item As described at the start of this proof, we have to remove the assumption that $\tau_{11}'=\tau_{11}$.
    		
    		Note that $\vspan{\tau_{22},\tau_{23},\dots,\tau_{2(\mu-1)},\tau_{2j'}}$ and $\vspan{\tau_{22},\tau_{23},\dots,\tau_{2(\mu-1)},\widetilde{\tau}_{2j'}}$ both span hyperplanes of $\vspan{\tau_{11}',\tau_{12},\dots,\tau_{1(\mu-1)},\tau_{1j'}}$ that do not contain $\tau_{11}'$, hence each of these hyperplanes intersect $\tau_{11}'$ in an $(\n-\rho-1)$-subspace $\overline{\tau}_{21}$ and $\widetilde{\tau}_{21}$, respectively, both containing $\HyperSub_2$.
    		
    		The goal is to find an $(\n-\rho-1)$-flower with $\mu+1$ petals such that the $j^\textnormal{th}$ petal contains $j$ concurrent, $\widetilde{\varphi}$-independent $(\n-\rho-1)$-dimensional $q'$-subgeometries contained in $\SubSat_{(\n,\rho)}\cup\HyperSub_2$, with the additional property that $P$ lies in the span of these petals, but not in the span of any $\mu$ petals.
    		It is clear that, if we can find an $(\n-\rho-1)$-subspace $\widehat{\tau}_{21}\notin\{\overline{\tau}_{21},\widetilde{\tau}_{21}\}$ in $\tau_{11}'$, not lying in $\Sigma_1$ and containing an $(\n-\rho-1)$-dimensional $q'$-subgeometry $\Sub\supseteq\HyperSub_2$, $\Sub\setminus\HyperSub_2\subseteq\SubSat_{(\n,\rho)}$, then $\{\widehat{\tau}_{21},\tau_{22},\dots,\tau_{2(\mu-1)},\tau_{2j'},\widetilde{\tau}_{2j'}\}$ is the $(\n-\rho-1)$-flower we are looking for.
    		\begin{itemize}
    		    \item If $q'>2$, then there exists an $(\n-\rho-1)$-dimensional subspace of $\Sub_{11}^{(1)}$ spanning an $(\n-\rho-1)$-space $\widehat{\tau}_{21}$ that contains $\Sigma_2$, but is not equal to $\Sigma_1$, $\overline{\tau}_{21}$ or $\widetilde{\tau}_{21}$.
    		    \item If $q'=2$, we distinguish the two possibilities described at the start of the proof:
    		    \begin{enumerate}[label=\roman*]
    		        \item suppose there exist at least two $(\n-\rho)$-dimensional $q'$-subgeometries $\Sub_1$ and $\Sub_2$ in $\tau_{11}'$, both containing $\HyperSub_1$ and a point $F\in\tau_{11}'\setminus\Sigma_1$, such that $\left(\Sub_1\cup\Sub_2\right)\setminus\HyperSub_1\subseteq\SubSat_{(\n,\rho)}$.
        		    By Lemma \ref{Lm_UniqueSubThroughHyperSubAndHyperplane}, we find at least three $(\n-\rho-1)$-spaces through $\Sigma_2$, not lying in $\Sigma_1$, that intersect either $\Sub_1$ or $\Sub_2$ in an $(\n-\rho-1)$-dimensional $q'$-subgeometry.
        		    Hence, one of these three $(\n-\rho-1)$-spaces $\widehat{\tau}_{12}$ cannot be equal to $\overline{\tau}_{21}$ or $\widetilde{\tau}_{21}$.
        		    \item if $\tau_{11}'=\tau_{11}$, then, by the definition of the set $\ppi_1'$, we can always find an $(\n-\rho-1)$-dimensional subspace $\widehat{\tau}_{21}$ with the desired properties.
    		    \end{enumerate}
    		\end{itemize}
		\end{enumerate}
		Intuitively, the steps above split the initial $(\n-\rho)$-flower with $\mu$ petals into an $(\n-\rho-1)$-flower with $\mu+1$ petals.
		For this new flower, the property that $P$ is contained in the span of all of its petals, but not in the span of any fewer number of petals, still holds.
		We execute the steps above a total of $\rho+1-\mu$ times, leaving us with an $(\n-2\rho+\mu-1)$-flower $\Flower_{\rho+2-\mu}'$ with $\rho+1$ petals.
		Note that this is always possible, as by the assumption corresponding to this case, $\rho+1-\mu\leq\lambda-1$, which means that, in each step, one can always choose smaller petals containing subgeometries (see Construction \ref{Constr_Subgeometric}) which fulfil the desired conditions.
		
		Moreover, for each $j\in\{1,\dots,\rho+1\}$, there must exist a petal in $\Flower_{\rho+2-\mu}'$ with $j$ concurrent, $\widetilde{\varphi}$-independent $(\n-2\rho+\mu-1)$-dimensional $q'$-subgeometries contained in $\SubSat_{(\n,\rho)}\cup\HyperSub_1$.
		Indeed, let $L_i$ be the list of numbers of concurrent, $\widetilde{\varphi}$-independent $(\n-\rho-i)$-dimensional $q'$-subgeometries we can find in the respective petals of the flower we obtain after going through the steps $i$ times ($i\in\{0,1,\dots,\rho+1-\mu\}$).
		Then, by considering the nature of the maps $\fn{\cdot}{\cdot}$ (see Definition \ref{Def_LowerHigherMaps}), we get
		\begin{align*}
		    L_0 &= (1,2,\dots,\mu-1,j')\textnormal{,}\\
		    L_1 &= (1,2,\dots,\mu-1,j',j'+1)\textnormal{,}\\
		    L_2 &= (1,2,\dots,\mu-1,j',j'+1,j'+2)\textnormal{,}\\
		    &\;\;\vdots\\
		    L_{\rho+1-j'} &= (1,2,\dots,\mu-1,j',j'+1,\dots,\rho+1)\textnormal{,}\\
		    L_{\rho-j'} &= (1,2,\dots,\mu-1,j'-1,j',j'+1,\dots,\rho+1)\textnormal{,}\\
		    &\;\;\vdots\\
		    L_{\rho+1-\mu} &= (1,2,\dots,\mu-1,\mu,\mu+1,\dots,j'-1,j',j'+1,\dots,\rho+1)\textnormal{.}\\
		\end{align*}
		Hence, Lemma \ref{Lm_FlowerSaturates} finishes the proof.
	\end{proof}
	
	\begin{thm}\label{Thm_MainUpperBoundSubgeometric}
	    Let $0<\rho<\n$ and let $q=(q')^{\rho+1}$ for any prime power $q'$.
	    Then
	    \begin{align*}
	        \satbound(\n,\rho)&\leq\sum_{i=1}^{k(\n,\rho)}\left(\frac{(\rho+1)(\rho+2)}{2}(q')^{\n+1-i(\rho+1)}\right)+\sum_{i=1}^{k(\n,\rho)-1}\sum_{j=1}^{\rho-1}\widetilde{a}(\rho,j)(q')^{\n+1-i(\rho+1)-j}\\
	        &+\sum_{j=1}^{\ell(\n,\rho)-1}\overline{a}(\n,\rho,j)(q')^{\ell(\n,\rho)-j}-\widetilde{c}(\n,\rho)-\overline{c}(\n,\rho)\\
	        &+\delta_{q'=2}\cdot\left((2^{\rho-1}-1)\cdot\sum_{i=1}^{k(\n,\rho)-1}\left(2^{\n-\rho+2-i(\rho+1)}\right)+2^{\ell(\n,\rho)}-2\right)\textnormal{,}
	    \end{align*}
	    with
	    \begin{itemize}
	        \item $k(\n,\rho):=\left\lceil\frac{\n-\rho}{\rho+1}\right\rceil$,
	        \item $\ell(\n,\rho):=\n+1-k(\n,\rho)\cdot(\rho+1)=\big(\n\pmod{\rho+1}\big)+1$,
	        \item $\widetilde{a}(\rho,j):=\frac{\rho(\rho+2j+1)-j(3j+1)}{2}\leq\frac{\rho(2\rho+1)}{3}$,
	        \item $\overline{a}(\n,\rho,j):=\frac{\ell(\n,\rho)\big(2\rho-\ell(\n,\rho)+2j+1\big)-j(3j+1)}{2}\leq \widetilde{a}(\rho,j)$,
	        \item $\widetilde{c}(\n,\rho):=\big(k(\n,\rho)-1\big)\frac{\rho^2(\rho+1)}{2}\geq0$,
	        \item $\overline{c}(\n,\rho):=\frac{\rho(\rho+1)+\ell(\n,\rho)\big(\ell(\n,\rho)-1\big)\big(2\rho-\ell(\n,\rho)+1\big)}{2}\geq0$,
	        \item $\delta_{q'=2}:=\begin{cases}
	            1\qquad&\textnormal{if }q'=2\textnormal{,}\\
	            0\qquad&\textnormal{if }q'\neq2\textnormal{.}
	        \end{cases}$
	    \end{itemize}
	\end{thm}
	\begin{proof}
	    First, let us assume $\n+1$ is a multiple of $\rho+1$.
	    Then $k(\n,\rho)=\frac{\n-\rho}{\rho+1}$ and $\ell(\n,\rho)=\rho+1$.
	    Moreover, if we interpret $\widetilde{a}(\rho,j)$ and $\overline{a}(\n,\rho,j)$ as quadratic polynomials in $j$ (with $j\in\{1,\dots,\rho-1\}$ and $j\in\{1,\dots,\rho\}$, respectively), then one can check that these polynomials reach their minimum values if $j\in\{1,\rho-1\}$ and $j\in\{1,\rho\}$, respectively.
	    Hence, as $\rho\geq1$, we have
	    \begin{equation}\label{Eq_ApproxTilde}
	        0\leq2\rho-2\leq\min\big\{\widetilde{a}(\rho,1),\widetilde{a}(\rho,\rho-1)\big\}\leq \widetilde{a}(\rho,j) 
	    \end{equation}
	    for every $j\in\{1,\dots,\rho-1\}$, and
	    \begin{equation}\label{Eq_ApproxBar}
	        0\leq\rho\leq\min\big\{\overline{a}(\n,\rho,1),\overline{a}(\n,\rho,\rho)\big\}\leq \overline{a}(\n,\rho,j)\textnormal{,}
	    \end{equation}
	    for every $j\in\{1,\dots,\rho\}$.
	    Furthermore, the fact that $\n+1$ is a multiple of $\rho+1$ together with $\rho<\n$ implies that $2\rho+1\leq\n$, hence we obtain $\widetilde{c}(\n,\rho)\leq\frac{1}{2}\rho(\rho+1)(\n-2\rho-1)$ and $\overline{c}(\n,\rho)=\frac{1}{2}\rho(\rho+1)^2$.
	    Also note that $k(\n,\rho)\geq1$.
	    If we denote with $B(\n,\rho)$ the upper bound described in this theorem, we can combine these results, together with \eqref{Eq_ApproxTilde} and \eqref{Eq_ApproxBar}, to conclude that
	    \begin{align*}
	        B(\n,\rho)&\geq\sum_{i=1}^{k(\n,\rho)}\left(\frac{(\rho+1)(\rho+2)}{2}(q')^{\n+1-i(\rho+1)}\right)-\frac{1}{2}\rho(\rho+1)(\n-2\rho-1)-\frac{1}{2}\rho(\rho+1)^2+\delta_{q'=2}\\
	        &=\sum_{i=1}^{k(\n,\rho)}\left((\rho+1)(q')^{\n+1-i(\rho+1)}\right)\\
	        &\qquad\qquad\qquad+\sum_{i=1}^{k(\n,\rho)}\left(\frac{1}{2}\rho(\rho+1)(q')^{\n+1-i(\rho+1)}\right)-\frac{1}{2}\rho(\rho+1)(\n-\rho)+\delta_{q'=2}\\
	        &\geq\sum_{i=1}^{k(\n,\rho)}\left((\rho+1)(q')^{\n+1-i(\rho+1)}\right)+\frac{1}{2}\rho(\rho+1)(q')^{\n-\rho}-\frac{1}{2}\rho(\rho+1)(\n-\rho)+\delta_{q'=2}\\
	        &=(\rho+1)\left(q^{k(\n,\rho)}+q^{k(\n,\rho)-1}+\dots+q\right)+\frac{1}{2}\rho(\rho+1)\left((q')^{\n-\rho}-(\n-\rho)\right)+\delta_{q'=2}\\
	        &\geq(\rho+1)\left(q^{k(\n,\rho)}+q^{k(\n,\rho)-1}+\dots+q\right)+(\rho+1)\geq\satbound(\n,\rho)\textnormal{.}
	    \end{align*}
	    The latter inequality follows from Corollary \ref{Crl_TrivialUpperBound}.
	    
		As a result, we can assume that $\n+1$ is not a multiple of $\rho+1$.
		
	    By Lemma \ref{Lm_SubgeometricIsSaturating}, we can choose a point set $\SubSat_{(\n,\rho)}$ in $\pg(\n,q)$ (described in Construction \ref{Constr_Subgeometric}) which $\rho$-saturates all points of $\pg(\n,q)$, except for the points contained in a certain $(\n-\rho-1)$-subspace $\Sigma$.
	    
	    If $\n-\rho-1\leq\rho$, then $\n-\rho-1<\rho$, as else $\n+1$ would be a multiple of $\rho+1$.
	    Hence, in this case, all points of $\Sigma$ are $\rho$-saturated by $\SubSat_{(\n,\rho)}$ as well, as we can simply choose $\rho+1$ points in $\ppi_1$ that span the subspace $\tau_{11}\supseteq\Sigma$.
	    
	    If $\n-\rho-1>\rho$, then, by Lemma \ref{Lm_SubgeometricIsSaturating}, we can choose a point set $\SubSat_{(\n-(\rho+1),\rho)}$ in $\Sigma$ which $\rho$-saturates all points of $\Sigma$, except for the points contained in a certain $\big(\n-2(\rho+1)\big)$-subspace of $\Sigma$.
	    We can repeat this process to obtain a union
	    \[
	        \SubSat_{(\n,\rho)}\cup\SubSat_{(\n-(\rho+1),\rho)}\cup\dots\cup\SubSat_{(\n-(k(\n,\rho)-1)(\rho+1),\rho)}
	    \]
	    of $k(\n,\rho)$ point sets that $\rho$-saturates all points of $\pg(\n,q)$.
	    
	    For each $i\in\{1,2,\dots,k(\n,\rho)\}$, the size of the point set $\SubSat_{\left(\n-(i-1)(\rho+1),\rho\right)}$ can be calculated using Lemma \ref{Lm_SizePii1} and Lemma \ref{Lm_SizeSumPiij}, where every instance of $\n$ has to be replaced by $\n-(i-1)(\rho+1)$, hence every instance of $\lambda$ has to be replaced by $\lambda_i:=\min\{\rho,\n-(i-1)(\rho+1)-\rho\}$.
	    \begin{itemize}
	        \item If $i\in\{1,2,\dots,k(\n,\rho)-1\}$, then $\rho<\n-(i-1)(\rho+1)-\rho$, which implies that $\lambda_i=\rho$.
	        \item If $i=k(\n,\rho)$, then $\rho\geq\n-\big(k(\n,\rho)-1\big)(\rho+1)-\rho$ (keeping in mind that $\n+1$ is no multiple of $\rho+1$), which implies that $\lambda_{k(\n,\rho)}=\ell(\n,\rho)$.
	    \end{itemize}
	    
	    Finally, we claim that $\overline{a}(\n,\rho,j)\leq \widetilde{a}(\rho,j)\leq\frac{\rho(2\rho+1)}{3}$, for all $j\in\{1,\dots,\rho\}$.
	    Indeed, for the first inequality, one can interpret $\frac{\ell(\n,\rho)\left(2\rho-\ell(\n,\rho)\right)}{2}$ as a quadratic polynomial in $\ell(\n,\rho)$, which reaches its maximum value if $\ell(\n,\rho)=\rho$.
	    For the second inequality, one can interpret $\widetilde{a}(\rho,j)$ as a quadratic polynomial in $j$, which reaches its maximum value if $j=\frac{2\rho-1}{6}$.
	    However, the latter is never an integer.
	    Hence, one can conclude that
	    \[
	        \widetilde{a}(\rho,j)\leq\max\left\{a\left(\rho,\frac{2\rho-1}{6}-\frac{1}{6}\right),a\left(\rho,\frac{2\rho-1}{6}+\frac{1}{6}\right)\right\}=\frac{\rho(2\rho+1)}{3}\textnormal{,}
	    \]
	    for all $j\in\{1,\dots,\rho\}$.
	\end{proof}
	
	As the upper bound presented in Theorem \ref{Thm_MainUpperBoundSubgeometric} is not easy to work with in practice, a simplified upper bound is desired.
	If $\rho$ is large enough, the upper bound of Theorem \ref{Thm_MainUpperBoundSubgeometric} simplifies considerably.
	More precisely, \textbf{if} $\boldsymbol{\rho\geq\frac{\n-1}{2}}$, then $k(\n,\rho)=1$ and $\ell(\n,\rho)=\n-\rho$, hence the bound of Theorem \ref{Thm_MainUpperBoundSubgeometric} becomes the following:
	\[
        \satbound(\n,\rho)\leq\frac{(\rho+1)(\rho+2)}{2}(q')^{\n-\rho}+\sum_{j=1}^{\n-\rho-1}\overline{a}(\n,\rho,j)(q')^{\n-\rho-j}-\overline{c}(\n,\rho)+\delta_{q'=2}\cdot\left(2^{\n-\rho}-2\right)\textnormal{.}
    \]
    
	In case $\rho>1$, one can deduce from Theorem \ref{Thm_MainUpperBoundSubgeometric} the following easy-to-read but slightly weaker bound, which we present as our main result.
	
	\begin{thm}\label{Thm_MainUpperBoundSubgeometricRhoLargerThan1}
	    Let $1<\rho<\n$ and let $q=(q')^{\rho+1}$ for any prime power $q'$.
		Then
		\[
		    \satbound(\n,\rho) \leq \frac{(\rho+1)(\rho+2)}{2}(q')^{\n-\rho} + \rho(\rho+1)\frac{(q')^{\n-\rho}-1}{q'-1}\textnormal{.}
		\]
	\end{thm}
	
	Translating the result above in coding theoretical terminology (see Subsection \ref{Subsec_CovCod}), one obtains the following.
	
	\begin{crl}\label{Crl_MainUpperBoundSubgeometricRLargerThan2}
	    Let $2<R<\red$ and let $q=(q')^R$ for any prime power $q'$.
	    Then
	    \[
		    \lenfunc(\red,R) \leq \frac{R(R+1)}{2}(q')^{\red-R} + (R-1)R\frac{(q')^{\red-R}-1}{q'-1}\textnormal{.}
		\]
		For any infinite family of covering codes of length equal to the upper bound above, the following holds for its asymptotic covering density:
		\[
		    \covdeninf(R)<\frac{\big((R-1)R\big)^R}{R!}\left(1+\frac{1}{q'}+\dots+\frac{1}{(q')^{R-1}}\right)^R<\left(e(R-1)\frac{q-1}{q-(q')^{R-1}}\right)^R\textnormal{,}
		\]
		with $e\approx2.718...$ being Euler's number.
	\end{crl}
	
	\begin{prob}
	    Given that $q=(q')^{\rho+1}=(q')^R$, improve the leading coefficient $\frac{(\rho+1)(\rho+2)}{2}$ of Theorem \ref{Thm_MainUpperBoundSubgeometric} and \ref{Thm_MainUpperBoundSubgeometricRhoLargerThan1} (respectively the leading coefficient $\frac{R(R+1)}{2}$ of Corollary \ref{Crl_MainUpperBoundSubgeometricRLargerThan2}) to one that is linear in $\rho$ (respectively linear in $R$).
	\end{prob}
	
	\textbf{Acknowledgements.}
	First of all, I would like to thank Fernanda Pambianco for her opinion on these results, as well as her guidance through the many literary works on this topic.
	Secondly, I would like to express my appreciation for the quick responses and helpful replies of Maarten De Boeck and Geertrui Van de Voorde on the topic of point-line geometries and linear representations.
	A special thanks goes towards Stefaan De Winter for his incredibly elegant sketch on how to construct an isomorphism between $Y(\rho,m,q')$ and $X(\rho,m,q')$ using field reduction (see Subsection \ref{Subsec_IsomorphismXFieldRed}), which motivated me to construct the direct isomorphism described in Section \ref{Sect_PointLine}.
	Finally, I would like to thank my supervisor Leo Storme and colleague Jozefien D'haeseleer for their time and effort in proofreading this work, as well as the anonymous referees for their insightful corrections and remarks.

\bibliographystyle{plain}
\bibliography{main}

\bigskip
Author's address:

\bigskip
Lins Denaux

Ghent University

Department of Mathematics: Analysis, Logic and Discrete Mathematics

Krijgslaan $281$ -- Building S$8$

$9000$ Ghent

BELGIUM

\texttt{e-mail : lins.denaux@ugent.be}

\texttt{website: }\url{https://users.ugent.be/~ldnaux}

\end{document}